\documentclass[12pt]{amsart}

\usepackage{amsmath,amssymb,amsthm}
\usepackage{color}
\usepackage[top=30truemm,bottom=30truemm,left=25truemm,right=25truemm]{geometry}
\usepackage{mathtools}
\usepackage{xypic}
\mathtoolsset{showonlyrefs=true}

\usepackage{url}
\usepackage{ulem}

\makeatletter
\@addtoreset{equation}{section}

\makeatother

\newcommand{\Z}{\mathbb{Z}}
\newcommand{\Q}{\mathbb{Q}}

\newcommand{\F}{\mathbb{F}}
\newcommand{\C}{\mathbb{C}}
\newcommand{\T}{\mathbb{T}}

\newcommand{\fq}{\mathfrak{q}}

\newcommand{\OO}{\mathcal{O}}

\newcommand{\GG}{{\mathcal{G}}}
\newcommand{\II}{\mathcal{I}}

\newcommand{\TT}{\mathcal{T}}
\newcommand{\RR}{\mathcal{R}}
\newcommand{\QQ}{\mathcal{Q}}

\newcommand{\dual}{\vee}
\newcommand{\ttilde}{\widetilde}
\newcommand{\Inv}{\mathcal{I}}
\newcommand{\derR}{\mathcal{R}}

\newcommand{\De}{D}

\newcommand{\DeP}{\De^{\perf}}
\newcommand{\DePT}{\De^{\perf}_{\tor}}
\newcommand{\Ch}{Ch}

\newcommand{\ChP}{\Ch^{\perf}}
\newcommand{\ChPT}{\Ch^{\perf}_{\tor}}

\DeclareMathOperator{\Gal}{Gal}
\DeclareMathOperator{\pd}{pd}
\DeclareMathOperator{\cyc}{cyc}
\DeclareMathOperator{\Ker}{Ker}
\DeclareMathOperator{\Cok}{Cok}
\DeclareMathOperator{\Image}{Im}
\DeclareMathOperator{\rank}{rank}
\DeclareMathOperator{\ord}{ord}
\DeclareMathOperator{\modif}{mod}

\DeclareMathOperator{\tor}{tor}
\DeclareMathOperator{\Det}{Det}
\DeclareMathOperator{\perf}{perf}
\DeclareMathOperator{\ram}{ram}
\DeclareMathOperator{\Fitt}{Fitt}
\DeclareMathOperator{\Hom}{Hom}
\DeclareMathOperator{\Frac}{Frac}
\DeclareMathOperator{\Spec}{Spec}
\DeclareMathOperator{\id}{id}
\DeclareMathOperator{\Cone}{Cone}
\DeclareMathOperator{\Min}{Min}
\DeclareMathOperator{\Ann}{Ann}
\DeclareMathOperator{\nat}{nat}

\let\oldenumerate\enumerate
\renewcommand{\enumerate}{
   \oldenumerate
   \setlength{\itemsep}{1pt}
   \setlength{\parskip}{0pt}
   \setlength{\parsep}{0pt}
}
\let\olditemize\itemize
\renewcommand{\itemize}{
   \olditemize
   \setlength{\itemsep}{1pt}
   \setlength{\parskip}{0pt}
   \setlength{\parsep}{0pt}
}

\theoremstyle{plain}
\newtheorem{thm}{Theorem}[section]
\newtheorem{lem}[thm]{Lemma}
\newtheorem{conj}[thm]{Conjecture}
\newtheorem{prob}[thm]{Problem}
\newtheorem{prop}[thm]{Proposition}
\newtheorem{cor}[thm]{Corollary}

\theoremstyle{definition}
\newtheorem{defn}[thm]{Definition}
\newtheorem{rem}[thm]{Remark}
\newtheorem{eg}[thm]{Example}

\title{Fitting ideals of $p$-ramified Iwasawa modules over totally real fields}

\author[C. Greither]{Cornelius Greither}
\address{Fakult\"at Informatik, Universit\"at der Bundeswehr M\"unchen,
85577 Neubiberg, Germany}
\email{cornelius.greither@unibw.de}

\author[T. Kataoka]{Takenori Kataoka}
\address{Faculty of Science and Technology, Keio University.
3-14-1 Hiyoshi, Kohoku-ku, Yokohama, Kanagawa 223-8522, Japan}
\email{tkataoka@math.keio.ac.jp}

\author[M. Kurihara]{Masato Kurihara}
\address{Faculty of Science and Technology, Keio University.
3-14-1 Hiyoshi, Kohoku-ku, Yokohama, Kanagawa 223-8522, Japan}
\email{kurihara@z7.keio.jp}

\begin{document}

\begin{abstract} We completely calculate the Fitting ideal of the
classical $p$-ramified Iwasawa module for any abelian extension
$K/k$ of totally real fields, using the shifted Fitting ideals
recently developed by the second author. This generalizes former
results where we had to assume that only $p$-adic places may ramify
in $K/k$. One of the important ingredients is the computation of some complexes
in appropriate derived categories.
\end{abstract}

\maketitle

{\it\small \tableofcontents\par}

\section*{Introduction}

One of the most important themes in Iwasawa theory is 
to study the relationship 
between $p$-adic analytic objects and $p$-adic algebraic objects,
usually formulated as ``main conjectures,'' 
in which the algebraic objects are described by characteristic 
ideals of suitable arithmetic modules.
However, we are recently understanding that 
there exists closer relationship between analytic 
and algebraic objects beyond characteristic ideals.
For example, such relationship can be described by using Fitting ideals.

Namely, in certain cases, 
using the $p$-adic $L$-functions corresponding to the arithmetic objects,  
we can describe the Fitting ideals of certain arithmetic modules,  
which give more information 
than the characteristic ideals.
But in such cases, {\it modified} Iwasawa modules had been always used; 
see for example, \cite{GP15}, \cite{BKS16}, etc.

In this paper we study a much more difficult and subtle 
object, the Fitting ideals of {\it non-modified classical} Iwasawa modules. 
We prove that they can be described by 
the analytic objects and certain ideals constructed from simple 
objects. 
We think it is remarkable that the Fitting ideals 
of classical Iwasawa modules can be also described by some variants of 
$p$-adic $L$-functions.

In order to explain slightly more details, we introduce here notation in this paper.
Throughout this paper, we fix an odd prime number $p$.
We consider a finite abelian extension $K/k$ of totally real 
number fields and the cyclotomic $\Z_p$-extension $K_\infty$ of $K$.
Let $S_p$ be the set of $p$-adic places of $k$.
For any algebraic extension $F/k$,
let $S_{\ram}(F/k)$ be the set of finite places of $k$ 
which are ramified in $F$.
For any finite set $S$ of primes of $k$, let $X_{K_\infty,S}$ be the $S$-ramified Iwasawa module, which is by definition the Galois group of the maximal  pro-$p$-abelian extension
of $K_{\infty}$ unramified outside $S$.
Recall that $X_{K_\infty,S}$ is a module over the Iwasawa algebra $\RR=\Z_p[[\GG]]$,
where $\GG = \Gal(K_{\infty}/k)$ is the profinite Galois group in
this setting.
We simply write $X_{S}$ for $X_{K_\infty,S}$ 
when no confusion arises.

The main theme in this paper is to compute $\Fitt_{\RR}(X_{S_p})$, the Fitting ideal of the Iwasawa module $X_{S_{p}}$.
The module $X_{S_p}$ has been important in 
Iwasawa theory.
For example,
we have the Kummer duality (see Remark \ref{rem:51}) between $X_{S_p}$ and $A_{K(\mu_{p^{\infty}})}^{\omega}$, the 
Teichm\"{u}ller character component of the inductive limit $A_{K(\mu_{p^{\infty}})}$ of the $p$-parts 
of the ideal class groups (full class groups) of $K(\mu_{p^{n}})$. 
The Fitting ideals of the minus part $(A_{K(\mu_{p^{\infty}})}^{-})^{\dual}$ 
are known outside the Teichm\"{u}ller character component 
(for example the method of the first author \cite{Greither07} can 
be applied to the Iwasawa theoretic situation without assuming 
the ETNC, see also \cite{Kur14} by the third author for the $T$-modified class 
group version and the Iwasawa theoretic version),
and 
we know that the Teichm\"{u}ller character component $A_{K(\mu_{p^{\infty}})}^{\omega}$, which is dual to our $X_{S_p}$, is a subtle and interesting 
object.
Note that this duality plays practically no role in the present paper.

In the papers \cite{GK15}, \cite{GK17} by the first and the third
author, and in the paper \cite{GKT19} with Tokio,
we determined $\Fitt_{\RR}(X_{S})$ when $S$ contains $S_{\ram}(K_{\infty}/k) = S_{\ram}(K/k) \cup S_p$.
Therefore, $\Fitt_{\RR}(X_{S_{p}})$ 
was determined in \cite{GK15}, \cite{GK17}, \cite{GKT19}
under the assumption that $S_{\ram}(K/k) \subset S_{p}$, that is, 
$K/k$ is unramified outside $p$.
But the assumption $S_{\ram}(K/k) \subset S_p$, 
is a pretty severe constraint.
In the present paper we completely remove the assumption $S_{\ram}(K/k) \subset S_p$, 
and determine $\Fitt_{\RR}(X_{S_p})$ for any finite abelian extension
$K/k$ of totally real fields. 
Thus we are mainly concerned with the case 
$S_{\ram}(K_{\infty}/k)  \supsetneq S_{p}$. 

The main result of this paper is the following.

\begin{thm}\label{thm:28}
Let $S$ be a finite set of finite places of $k$ such that 
$S \supset S_p \cup S_{\ram}(K/k)$ and $S \neq S_p$.
Put $S'=S \setminus S_{p} \neq \emptyset$.
Suppose the $\mu$-invariant of $X_{S_p}$ vanishes.
Then we have
\[
\Fitt_{\RR}(X_{S_p}) = \Fitt^{[1]}_{\RR}(Z_{S'}^0) \, \theta_{S}^{\modif}.
\]
\end{thm}

The definitions of the $\RR$-module $Z_{S'}^0$, of the element 
$\theta_{S}^{\modif}$, and of $\Fitt_{\RR}^{[1]}$ will be 
given in Section \ref{sec:46}.
We introduce in this paper an integral element 
$\theta_{S}^{\modif} \in \RR$, 
which is a kind
of (modified) equivariant Iwasawa power series.
The shifted Fitting ideal $\Fitt_{\RR}^{[1]}$ was introduced 
by the second author in \cite{Kata}. 
It is defined by using
a certain type of 
resolutions and the syzygies produced by them. The main point 
of the theorem is that all quantities on the right hand side are 
computable in principle.

The crucial point in this study
is the case when $\Gal(K/k)$ is a $p$-group.
In fact, the ring $\RR=\Z_p[[\GG]]$ is semi-local, 
and decomposed into direct product of local rings. 
Let $\GG=\GG^{(p')} \times \GG^{(p)}$ be the decomposition of $\GG$ 
such that $\GG^{(p')}$ is a finite group of order prime to $p$ and
$\GG^{(p)}$ is a pro-$p$ group. 
Then each local component of $\RR$ corresponds to an equivalence class of 
characters of $\GG^{(p')}$ (see Subsection \ref{subsecDecGroupRing}),
and accordingly the statement of Theorem \ref{thm:28} can be decomposed.
On the one hand, the trivial character component is the most difficult, and the statement is equivalent to that for the pro-$p$ extension $(K_{\infty})^{\GG^{(p')}}/k$ with Galois group $\GG^{(p)}$.
On the other hand, the non-trivial character components are easier to handle; for example, those components of $\Fitt^{[1]}_{\RR}(Z_{S'}^0)$ can be computed easily (see Corollary \ref{Non-trivialCharacterFittingIdeal}).
In that sense 
the case that $\Gal(K/k)$ is a $p$-group is essential.
However, the proof of Theorem \ref{thm:28} does not involve an explicit reduction
to that case.

The proof of our main result
occupies Sections \ref{sec:60} and \ref{sec:76}; indeed the proof splits naturally
into an algebraic part and an arithmetic part. The former constructs
a certain complex $C_S$ via an exact triangle, 
whose other two terms
come from complexes that arise in
global and local Galois cohomology respectively.
This produces a short exact sequence 
\[
0 \to X_{S_p} \to H^1(C_S)
\to Z^0_{S'} \to 0,
\]
as in Proposition \ref{prop:26}.
Since the middle term turns out to
be cohomologically trivial, this already gives a formula
for  $\Fitt_{\RR}(X_{S_p})$ in terms of $\Fitt^{[1]}_{\RR}(Z_{S'}^0)$: these two quantities differ by a principal factor
governed by the complex $C_S$. 
In the second part of the proof, this
factor is then identified with the (equivariant, modified)
$p$-adic $L$-function $\theta_{S}^{\modif}$.

In Subsection \ref{subsec:61}, we also discuss the natural question under 
what circumstances
$\Fitt_{\RR}(X_{S_p})$ is principal. The rough answer is: 
very rarely (see Proposition \ref{prop:63}).

In Section 4 we present several attempts to make our determination
of $\Fitt_{\RR}(X_{S_p})$ really explicit. The module $Z^0_{S'}$
that occurs in the main result appears to be fairly explicit, 
but a closer look quickly shows that (unless the 
extension $K/k$ is very small in a way) an honestly explicit
description of its first shifted Fitting ideal is
not obvious at all, and in fact turns out to be pretty hard in general.
We present a general method to attack the problem, and show that
it produces in some nice cases a truly explicit result, that
is, a concrete list of generators for $\Fitt^{[1]}(Z^0_{S'})$. 

In the final Section 5 we 
compute $\Fitt^{[1]}_{\RR}(Z_{S'}^0)$
explicitly to determine the Fitting ideal of $X_{S_{p}}$ in the case that 
$K/k$ is {\it cyclic} and satisfies some mild conditions (see Theorems \ref{thm:46} and \ref{thm:47}).
Especially, these results give generalizations of the main result in \cite{Kur11}
by the third author where only the case that $K/k$ is of degree $p$ was treated.
We think that this new look at the third author's previous result is a good way 
to use our main result and to test the techniques of Section 4.

\begin{rem}
Large parts of this paper, as they are
written now, make an essential use of homological algebra.
More precisely speaking, we need
the theory of complexes including the cone construction, and
some theory of derived categories.
We would like to mention here that in
the earliest stages of this manuscript we used different and
more elementary methods. Actually, as far as the proof
of the main result is concerned,  one might
call those other methods old-fashioned, since they mimicked
and partially repeated ingenious arguments of Tate
\cite{Ta66}, which are over fifty years old.
It is interesting to note that already in those old arguments 
one can perceive some central ideas of homological algebra like
the use of Ext groups, but the theory of complexes was not used in the way 
we know it today. Anyway,
it may be reassuring to know that alternative arguments exist,
but we think that using the framework of Galois cohomology
and complexes leads to shorter arguments
and to a better logical structure, so this is what the reader will
actually see in the body of this paper.
\end{rem}


\section{Ingredients for the main result}\label{sec:46}

Our main result in this paper is Theorem \ref{thm:28} 
in Introduction. In this section we define the $\RR$-module $Z_{S'}^0$, 
the element $\theta_S^{\modif} \in \RR$, and 
$\Fitt^{[1]}_{\RR}$ which appeared in the statement of Theorem \ref{thm:28}, 
and also give detailed explanation of several statements mentioned 
in Introduction.

We recall some important notation from the introduction.

Let $p$ be an odd prime number, $K/k$ a finite abelian extension of totally real fields, and $K_{\infty}$ the cyclotomic $\Z_p$-extension of $K$.
Put $\GG = \Gal(K_{\infty}/k)$ and $\RR = \Z_p[[\GG]]$.
We denote by $S_p$ the set of $p$-adic primes of $k$, and by $S_{\ram}(K/k)$ the set of primes of $k$ which are ramified in $K/k$.
Let $X_{S_p} = X_{S_p}(K_{\infty})$ be the $S_p$-ramified Iwasawa module for $K_{\infty}$.

\subsection{Definition of $Z_{S'}^0$} \label{subsection11}

As in Theorem \ref{thm:28}, let $S$ be a finite set of finite places of $k$ such that $S \supset S_p \cup S_{\ram}(K/k)$ and $S \neq S_p$. Put $S' = S \setminus S_p \neq \emptyset$.

For each finite place $v$ of $k$ outside $p$, let $\GG_v$ be the 
decomposition subgroup of $\GG$ at $v$.
Then $\GG_v$ is an open subgroup of $\GG$. Put
\[
Z_v = \Z_p[\GG/\GG_v],
\]
which is regarded as an $\RR$-module; note 
that it is a finitely generated
free $\Z_p$-module. Moreover, put
\[
Z_{S'} = \bigoplus_{v \in S'} Z_v.
\]
Finally, define an $\RR$-module $Z_{S'}^0$ by the exact sequence
\begin{equation}\label{eq:23}
0 \to Z_{S'}^0 \to Z_{S'} \to \Z_p \to 0,
\end{equation}
where the map $Z_{S'} \to \Z_p$ is defined to be the augmentation map
on each summand $Z_v$.  Note that this map is onto for the
precise reason that we assume $S'$ to be nonempty.

\subsection{Definition of $\theta_{S}^{\modif}$}

Again let $S$ be a finite set of finite places of $k$ such that $S \supset S_p \cup S_{\ram}(K/k)$, but we do not assume $S \neq S_p$ in this subsection.

\begin{defn}\label{defn:74}
Let $v$ be a finite place of $k$ outside $p$.
We denote by $N(v)$ the cardinality of the residue field of $k$ at $v$.
Let $\TT_v \subset \GG_v$ be the inertia group, which is finite.
Let $\sigma_v \in \GG/\TT_v$ be the $N(v)$-th power Frobenius automorphism.
\end{defn}

\begin{defn}
For a finite character $\psi: \GG = \Gal(K_{\infty}/k) \to \C^{\times}$, we have 
the $S$-imprimitive $L$-function 
\[
  L_S(s, \psi) = \prod_{v \not\in S} \left(1 - \frac{\psi(\sigma_v)}
    {N(v)^s} \right)^{-1},
\]
where $v$ runs over the finite places of $k$ that are not in $S$.
This infinite product converges on the half plane $\Re(s) > 1$ and 
$L_S(\psi, s)$ has a meromorphic continuation to the whole plane $\C$.
\end{defn}

We fix embeddings $\overline{\Q} \hookrightarrow \overline{\Q_p}$ 
and $\overline{\Q} \hookrightarrow \C$.
Then each finite character $\psi: \GG \to \C^{\times}$ 
can be regarded to have values in $\overline{\Q_p}^{\times}$.
Thus $\psi$ induces a continuous $\Z_p$-algebra homomorphism 
$\RR = \Z_p[[\GG]] \to \overline{\Q_p}$, which we again denote by $\psi$.

Let 
\[
\kappa_{\cyc}: \Gal(k(\mu_{p^{\infty}})/k) \hookrightarrow \Z_p^{\times}
\]
denote the cyclotomic character, and
\[
\omega: \Gal(k(\mu_p)/k) \hookrightarrow \Z_p^{\times}
\]
denote the Teichm\"uller character.
The $\Z_p$-algebra homomorphisms induced by them 
are also written by the same letters.

The $S$-truncated
$p$-adic $L$-functions $\theta_{S}$ are defined via interpolation properties, as follows.

\begin{defn}\label{defn:44} 
Let $\theta_S = \theta_{S, K_{\infty}/k} \in \Frac(\RR)^{\times}$ be the element characterized by
\[
(\kappa_{\cyc}^n\psi)(\theta_S) = L_S(1 - n, \psi \omega^{-n})
\]
for each finite character $\psi$ of $\GG$ and positive integers $n$ 
with $n \equiv 0 \bmod [K(\mu_p):K]$.
The existence of $\theta_S$ follows from Deligne-Ribet \cite{DR80}.
Moreover, it is known that $\theta_S$ is a pseudo-measure 
in the sense of Serre \cite{Serre78}, 
that is, we have $\Ann_{\RR}(\Z_p) \theta_S \subset \RR$.
\end{defn}

We define a modification $\theta_S^{\modif}$ of $\theta_S$ as follows.

\begin{defn}\label{defn:67}
We define $\theta_S^{\modif}= \theta_{S, K_{\infty}/k}^{\modif} 
\in \Frac(\RR)^{\times}$ by the interpolation formula 
\[
  (\kappa_{\cyc}^n\psi)(\theta_S^{\modif}) = L_S(1 - n, \psi \omega^{-n}) 
  \prod_{v} \frac{1 - \psi(\sigma_v)N(v)^n}{1 - \psi(\sigma_v)N(v)^{n-1}}
\]
for $\psi$ and $n$ as in Definition \ref{defn:44}, where $v$ runs 
over the elements in $S'$ such that $\psi$ is unramified at $v$.
\end{defn}

It is easy to see that $\theta_S^{\modif}$ is
uniquely determined by the interpolation properties.
It is also easy to deduce the existence of $\theta_S^{\modif}$ from 
that of $\theta_S$ 
as an element of $\Frac(\RR)^{\times}$. Furthermore,  
we show later in Subsection \ref{Integralityoftheta} the following.

\begin{thm} \label{PropInt} Our modified $p$-adic $L$-function 
$\theta_S^{\modif}$ is integral, namely
$\theta_S^{\modif} \in \RR$.
\end{thm}

We note that 
a variant of this element $\theta_{S}^{\rm mod}$ was called 
``Greither's Stickelberger element'' in \cite[Theorem 0.1]{Kur11}, 
and its integrality was proved in \cite[Lemma 2.1]{Kur11} for a special 
type of extension $K/k$ studied there.

We also note that we do not use Theorem \ref{PropInt} in 
the proof of Theorem \ref{thm:28}.

Let us first discuss some basic properties of $\theta_S^{\modif}$.

\begin{lem}\label{int2}
(1) Let $S_1$ be a finite set which contains $S$. 
Then we have 
\[ \theta_{S_1}^{\rm mod} = 
\theta_S^{\rm mod} \prod_{v \in S_1 \setminus S} (1 - \sigma_v).  \]

(2) 
Recall that, for any intermediate field $M_{\infty}$ of $K_{\infty}/k_{\infty}$, 
we also have an element $\theta_{S, M_{\infty}/k}^{\modif} \in \Frac(\Z_p[[\Gal(M_{\infty}/k)]])$ as in Definition \ref{defn:67}.
Then the image of $\theta_{S, K_{\infty}/k}^{\modif}$ in $\Frac(\Z_{p}[[\Gal(M_{\infty}/k)]])$ coincides with 
$\theta_{S, M_{\infty}/k}^{\modif}$.
\end{lem}

\begin{proof}
(1) We note that $v \in S_{1} \setminus S$ is unramified in $K_{\infty}/k$ 
because $S$ contains all ramifying primes. 
By the definition of $\theta_{S}$, we have 
\[  \theta_{S_1} = \theta_S \prod_{v \in S_1 \setminus S} 
  (1 - \sigma_v N(v)^{-1}). \]
Then by the definition of $\theta_S^{\modif}$, we obtain
\begin{align}
\theta_{S_1}^{\modif} 
    &= \theta_S^{\modif} \prod_{v \in S_1 \setminus S} 
\left[ (1 - \sigma_v N(v)^{-1}) \cdot \frac{1 - \sigma_v}{1 - \sigma_v N(v)^{-1}} \right]  \\
    &= \theta_S^{\modif} \prod_{v \in S_1 \setminus S} (1 - \sigma_v).
\end{align}

The claim in (2) follows from the interpolation properties of $\theta_{S, K_{\infty}/k}^{\modif}$ and $\theta_{S, M_{\infty}/k}^{\modif}$.
\end{proof}

     \subsection{Shifted Fitting ideals}

We review the theory of the second author \cite{Kata} on Fitting 
invariants. Let $\pd_{\RR}(P)$ be the projective dimension of an 
$\RR$-module $P$. By \cite[Theorem 2.6]{Kata} and 
\cite[Proposition 2.7]{Kata}, we have the following.

\begin{thm}
Let $n$ be a non-negative integer and $X$ a finitely generated 
torsion $\RR$-module. Take an $n$-step
resolution $0 \to Y \to P_1 \to \dots \to P_n \to X \to 0$ of $X$,
in which all modules are finitely generated torsion over $\RR$
and such that $\pd_{\RR}(P_i) \leq 1$ for $i=1,\ldots,n$.
If we put
\[
\Fitt_{\RR}^{[n]}(X) = \left( \prod_{i=1}^n \Fitt_{\RR}(P_i)^{(-1)^i} \right) \Fitt_{\RR}(Y),
\]
then the fractional ideal $\Fitt_{\RR}^{[n]}(X)$ of $\RR$ is 
independent of the choice of the 
$n$-step resolution.
In this sense, $\Fitt_{\RR}^{[n]}(X)$ is well defined.
\end{thm}

 \subsection{Decomposition of group rings} \label{subsecDecGroupRing} 

In general, suppose that $\Delta$ is a finite abelian group of order prime to $p$. 
Then we have a decomposition
$$\Z_{p}[\Delta] \simeq \prod_{\chi} {\mathcal O}_{\chi},$$
%
where $\chi$ runs over equivalence classes of $p$-adic 
characters of $\Delta$ (two characters $\chi_1$, $\chi_2$ are equivalent 
if and only if $\sigma \chi_1=\chi_2$ for some $\sigma \in 
\Gal(\overline{\Q}_{p}/\Q_{p})$), and 
${\mathcal O}_{\chi}=\Z_{p}[\Image(\chi)]$ is a $\Z_{p}[\Delta]$-module 
on which $\Delta$ acts via $\chi$. 
According to this decomposition, each $\Z_p[\Delta]$-module $M$ can be decomposed as 
\[
M = \bigoplus_{\chi} M^{\chi}
\]
with $\OO_{\chi}$-modules $M^{\chi}$.

Now we consider $\GG=\Gal(K_{\infty}/k)$. 
We decompose it into $\GG=\GG^{(p')} \times \GG^{(p)}$ where 
$\GG^{(p')}$ is a finite group of order prime to $p$ and
$\GG^{(p)}$ is a pro-$p$ group.
Since $\Z_{p}[[\GG]]=\Z_{p}[\GG^{(p')}][[\GG^{(p)}]]$, 
applying the above 
decomposition of $\Z_{p}[\Delta]$ to $\Delta=\GG^{(p')}$, we have  
$$\RR=\Z_{p}[[\GG]] \simeq \prod_{\chi} {\mathcal O}_{\chi}[[\GG^{(p)}]].$$
We also have
$$\RR=\Z_{p}[[\GG]] \simeq \Z_{p}[[\GG^{(p)}]] \times
\prod_{\chi \neq 1} {\mathcal O}_{\chi}[[\GG^{(p)}]]$$
where the first component of the right hand side corresponds to the trivial character $\chi=1$.

Here we give a description of $\Fitt_{\RR}^{[1]}(Z_v)$.
Let $v$ be a finite place of $k$ outside $p$.
Recall (Definition \ref{defn:74}) that $\TT_v$ is the inertia group in $K_{\infty}/k$, and $\sigma_v \in \GG/\TT_v$ is the Frobenius automorphism.
Let 
\[
\nu_{v}: \Frac(\Z_p[[\GG/\TT_v]]) 
\to \Frac(\Z_p[[\GG]])
\]
 be the map
induced by the multiplication by the norm element
$N_{\TT_v}=\sum_{\sigma \in \TT_v} \sigma$.

\begin{prop}\label{prop:85}
For each finite place $v$ of $k$ outside $p$, we have
\[
\Fitt_{\RR}^{[1]}(Z_v) = \left(1, \nu_v \frac{1}{\sigma_v-1} \right)
\]
as fractional ideals of $\RR$.
\end{prop}

\begin{proof}
The statement of the lemma can be decomposed according to characters 
$\chi$ of $\GG^{(p')}$.
We consider the decomposition $\TT_v = \TT_v^{(p')} \times \TT_v^{(p)}$ 
where the order of $\TT_v^{(p')}$ is prime to $p$ and $\TT_v^{(p)}$ is a pro-$p$ group.
If $\chi$ is non-trivial on $\TT_v^{(p')}$, 
we have $Z_v^{\chi} = 0$ and $\chi(\nu_v) = 0$, so the equation holds.
Therefore, we only have to deal with $\chi$ which is trivial on $\TT_v^{(p')}$. 
Thus we may assume $\TT_v = \TT_v^{(p)}$ from the start.

Assume $\TT_v = \TT_v^{(p)}$.
Note that, by local class field theory, $\TT_v$ is a quotient of the unit group $\OO_{k_v}^{\times}$, so in particular $\TT_v^{(p)}$ is a cyclic group.
Hence we can take a generator $\delta_v$ of $\TT_v$.
Take a lift $\ttilde{\sigma_v} \in \GG$ of $\sigma_v \in \GG/\TT_v$.
Then $\GG_v$ is topologically generated by $\ttilde{\sigma_v}$ and $\delta_v$, so we have $Z_v \simeq \RR/(\ttilde{\sigma_v} - 1, \delta_v-1)$.
Thus we have an exact sequence
\[
\RR/(\ttilde{\sigma_v} - 1) \overset{N_{\TT_v}}{\to} \RR/(\ttilde{\sigma_v} - 1) \overset{\delta_v - 1}{\to}
\RR/(\ttilde{\sigma_v} - 1) \to Z_v \to 0.
\]
Observe that the cokernel of $N_{\TT_v}$ here has a presentation $(N_{\TT_v}, \ttilde{\sigma_v}-1)$ as an $\RR$-module.
Hence we obtain
\[
\Fitt_{\RR}^{[1]}(Z_v) = (\ttilde{\sigma_v}-1)^{-1} (N_{\TT_v}, \ttilde{\sigma_v}-1)
= \left(1, \nu_v \frac{1}{\sigma_v-1}\right).
\]
\end{proof}

Now we give an explicit description of $\Fitt^{[1]}_{\RR}(Z_{S'}^0)$ 
for {\it non-trivial} character components.

\begin{prop} \label{PropForNon-trivialCharacter}
For any non-trivial character $\chi$ of $\GG^{(p')}$, we have
$$
\Fitt^{[1]}_{\RR^{\chi}}((Z_{S'}^0)^{\chi})=
\prod_{v \in S'}
\left(1, \nu_v \frac{1}{\sigma_{v}-1} \right)
$$
as fractional ideals of $\RR^{\chi}$.
\end{prop}

\begin{proof}
Since $\chi$ is non-trivial, we have $(\Z_p)^{\chi} = 0$, so $(Z_{S'}^0)^{\chi} = (Z_{S'})^{\chi}$.
Then the assertion follows from Proposition \ref{prop:85} immediately.
\end{proof}

Note that $N_{\TT_v} \in \RR$ goes to $0$ in $\RR^{\chi}$ unless $\chi$ is trivial on $\TT_v$, that is, $\chi(v) = 1$.
Therefore, only places $v \in S'$ with $\chi(v) = 1$ contribute in the product.

Using Theorem \ref{thm:28} and the Proposition \ref{PropForNon-trivialCharacter}, we get 
a {\it complete description} of the Fitting ideal of the non-trivial character 
component of $X_{S_{p}}$.

\begin{cor} \label{Non-trivialCharacterFittingIdeal}
For any non-trivial character $\chi$ of $\GG^{(p')}$, we have
\[
\Fitt_{\RR}(X_{S_p}^{\chi}) = \prod_{v \in S'}
\left(1, \nu_{v}\frac{1}{\sigma_{v}-1} \right) \, (\theta_{S}^{\modif})^\chi.
\]
\end{cor}


\section{Proof of main result (I)}\label{sec:60}

     \subsection{Some facts on arithmetic complexes}

We collect some facts on local and global arithmetic complexes.
A comprehensive reference is Nekov\'{a}\v{r} \cite{Nek}.

Let $k_S/k$ be the maximal $S$-ramified algebraic extension.
For each finite place $v$ of $k$, let $k_v$ be the completion at $v$.
Fix an algebraic closure $\overline{k_v}$ of $k_v$ and an inclusion $k_S \hookrightarrow \overline{k_v}$ over $k$.
Then any representation of $\Gal(k_S/k)$ 
will yield a representation of $\Gal(\overline{k_v}/k_v)$.

We denote by $(-)^{\dual}$ the Pontryagin dual of a module.
This symbol will also be used for the corresponding
construction in derived categories.
As usual, we denote by $\mu_{p^m}$ the group of $p^m$-th roots of unity.
Let $\Z_p(1) = \varprojlim_m \mu_{p^m}$ be the Tate module.
Let  $\chi_{\GG}: \Gal(k_S/k) \twoheadrightarrow \Gal(K_{\infty}/k) = \GG \hookrightarrow \RR^{\times}$ be the tautological representation.
We consider
\[
\T = \Z_p(1) \otimes_{\Z_p} \RR(\chi_{\GG}^{-1}),
\]
which is an $\RR$-module of rank one with a certain action of $\Gal(k_S/k)$.

We shall study the complexes
\begin{equation}\label{eq:24}
\derR\Gamma(k_S/k, \T), \qquad \derR\Gamma(k_S/k, \T^{\dual}(1))^{\dual}, 
  \qquad 
\derR\Gamma(k_v, \T), \qquad \derR\Gamma(k_v, \T^{\dual}(1))^{\dual},
\end{equation}
which are defined using the continuous cochain complexes for the profinite
groups $\Gal(k_S/k)$ and $\Gal(\overline{k_v}/k_v)$ 
(see \cite[(3.4.1)]{Nek}).

We denote by $\DeP(\RR)$ the derived category of perfect complexes 
of $\RR$-modules, and by $\DePT(\RR)$ the subcategory of objects 
whose cohomology groups are torsion as $\RR$-modules.
We will see that most of the complexes we treat in this paper are 
objects of $\DePT(\RR)$. First we recall the following fact.

\begin{prop}[{\cite[Proposition (4.2.9)]{Nek}}]\label{prop:39}
The ``global complex''
\[
\derR\Gamma(k_S/k, \T),
\]
as well as the ``local complexes''
\[
\derR\Gamma(k_v, \T)
\]
for any finite place $v$ of $k$, are objects of $\DeP(\RR)$.
\end{prop}

The following two propositions are interpretations of the 
local Tate duality and the global 
Poitou-Tate exact sequence, respectively.
It would be certain that they have been known 
since Grothendieck's works, and are explicitly mentioned in 
Nekov\'{a}\v{r} \cite{Nek}.
We use \cite[(2.9.1)]{Nek} to identify the Pontryagin dual 
and the Matlis dual \cite[(2.3)]{Nek}.

\begin{prop}[{\cite[Proposition (5.2.4)(i)]{Nek}}]\label{prop:40}
We have an isomorphism
\[
\derR\Gamma(k_v, \T) \simeq \derR\Gamma(k_v, \T^{\dual}(1))^{\dual}[-2].
\]
\end{prop}

\begin{prop}[{\cite[Proposition (5.4.3)(i)]{Nek}}]\label{prop:41}
We have a distinguished triangle
\[
\derR\Gamma(k_S/k, \T) \to \bigoplus_{v \in S} \derR\Gamma(k_v, \T) 
\to \derR\Gamma(k_S/k, \T^{\dual}(1))^{\dual}[-2] \to,
\]
where the first morphism is obtained by the localization, and
the second morphism by the localization and the duality in Proposition \ref{prop:40}.
\end{prop}

\begin{rem} 
It should also be possible to deduce this exact triangle from
the paper \cite{BuFl98}. The notation there is closer in spirit to ours
than Nekov\'{a}\v{r}'s,
but there is the disadvantage that everything is formulated at finite
level, and we have not checked whether the transition to the 
projective limit offers
problems. A little more precisely: The definition of the cone in \cite{BuFl98}, 
formula (3) on p.1345, gives an exact triangle
\[
\derR\Gamma(k_S/k, \Z_p(1)) \to \bigoplus_{v \in S} \derR\Gamma(k_v, \Z_p(1)) 
\to \mathcal C,
\]
where $\mathcal C$ denotes the cone. Then with the method of loc.cit.
p.1357, see equation (36) in particular, it should be possible
to identify $\mathcal C$ with 
$\derR\Gamma(k_S/k, \Z_p(1)^\vee(1))^\vee [-2]$. Again,
we gloss over some technical problems and we do not try to discuss the
passage from $\Z_p(1)$ (finite level) to $\T$ (infinite level).
\end{rem}

Next we compute the cohomology groups of the global and local complexes.

\begin{defn}
Let $v$ be a finite place of $k$ outside $p$.
Put
\[
J_v = J_v(K_{\infty}) = \varprojlim_n \mu_{p^{\infty}}(K_n \otimes_k k_v),
\]
where $\mu_{p^{\infty}}(K_n \otimes_k k_v)$ denotes the $p$-primary 
subgroup of $(K_n \otimes_k k_v)^{\times}$ and the inverse limit 
is taken with respect to the norm maps.
Then $J_v$ is naturally an $\RR$-module and its structure 
is as in Remark \ref{rem:37}. Put 
\[
J_{S'} = \bigoplus_{v \in S'} J_v.
\]
\end{defn}

Let $X_S = X_S(K_{\infty})$ be the $S$-ramified Iwasawa module.
From global class field theory, we have an exact sequence
\begin{equation}\label{eq:22}
0 \to J_{S'} \to X_S \to X_{S_p} \to 0
\end{equation}
where the injectivity of $J_{S'} \to X_{S}$ 
follows from the weak Leopoldt conjecture.

\begin{rem}\label{rem:37}
Take a place $w$ of $K_{\infty}$ above $v$, and put
\[
J_w = J_w(K_{\infty}) = \varprojlim_n \mu_{p^{\infty}}(K_{n, w}).
\]
Here $K_{n,w}$ denotes the completion of $K_n$ at the place below $w$.
Then we have $J_v \simeq \RR \otimes_{\RR_v} J_w$, where 
$\RR_v = \Z_p[[\GG_v]]$.

If $\mu_{p^{\infty}}(K_{\infty, w}) = 0$, then we have $J_w = 0$ 
and thus $J_v = 0$. Otherwise, we have $\mu_{p^{\infty}} 
\subset (K_{\infty, w})^{\times}$ and $J_w \simeq \Z_p$.
In the latter case, the action of $\GG_v$ on $J_w$ is given 
by the cyclotomic character $\kappa_v: \GG_v \to \Z_p^{\times}$ 
at $v$, and we have $J_v \simeq \Z_p[\GG/\GG_v]$ as a $\Z_p$-module.
\end{rem}

\begin{prop}\label{prop:35}
We have
\[
H^i(k_S/k, \T^{\dual}(1))^{\dual} \simeq 
\begin{cases}
	X_S & (i = 1)\\
	\Z_p & (i = 0)\\
	0 & (i \neq 0, 1)
\end{cases}
\]
and
\[
H^i(k_v, \T) \simeq 
\begin{cases}
	 J_v & (i=1)\\
	Z_v & (i=2)\\
	0 & (i \neq 1, 2)
\end{cases}
\]
for $v \nmid p$
where $Z_v$ was defined in Subsection \ref{subsection11}.
\end{prop}

\begin{proof} We feel that it should also be possible to assemble
a proof from suitable references to Nekov\'{a}\v{r}'s book \cite{Nek}, but we will 
write out a direct proof for the reader's convenience.

We have
\[
\T = \Z_p(1) \otimes_{\Z_p} \RR(\chi_{\GG}^{-1}) 
\simeq \varprojlim_{n} \left( \Z_p(1) 
   \otimes_{\Z_p} \Z_p[\Gal(K_n/k)](\chi_{\GG_n}^{-1}) \right),
\]
where $\chi_{\GG_n}: \Gal(k_S/k) \twoheadrightarrow \Gal(K_n/k) 
\hookrightarrow \Z_p[\Gal(K_n/k)]^{\times}$ is the tautological 
representation.  Then
\begin{align}
H^i(k_S/k, \T^{\dual}(1))
&\simeq \varinjlim_n H^i(k_S/k, (\Z_p(1) \otimes_{\Z_p} 
  \Z_p[\Gal(K_n/k)](\chi_{\GG_n}^{-1}))^{\dual}(1))\\
&\simeq \varinjlim_n H^i(k_S/k, (\Q_p/\Z_p) \otimes_{\Z_p} 
  \Z_p[\Gal(K_n/k)](\chi_{\GG_n}))\\
&\simeq \varinjlim_n H^i(k_S/K_n, \Q_p/\Z_p)\\
&\simeq H^i(k_S/K_{\infty}, \Q_p/\Z_p),
\end{align}
where the third isomorphism follows from Shapiro's lemma.
The weak Leopoldt conjecture, which says that  
$H^2(k_S/K_{\infty}, \Q_p/\Z_p)$
vanishes, is known to be true. This implies the first assertion
of Proposition \ref{prop:35}.

For the second assertion, we use Proposition \ref{prop:40} to see 
that $H^i(k_v, \T) \simeq  H^{2-i}(k_v, \T^{\dual}(1))^{\dual}$.
Take a place $w$ of $K_{\infty}$ above $v$.
A computation similar to the global case that we just have done 
shows ($G_{v,n}$ is an ad hoc abbreviation for $\Gal(K_{n, w}/k_v)$):
\begin{align}
H^i(k_v, \T^{\dual}(1))^{\dual}
&\simeq \left[ \varinjlim_n H^i(k_v, (\Z_p(1) \otimes_{\Z_p} 
  \Z_p[\Gal(K_n/k)](\chi_{\GG_n}^{-1}))^{\dual}(1)) \right]^{\dual}\\
&\simeq \left[ \varinjlim_n H^i(k_v, (\Q_p/\Z_p) \otimes_{\Z_p} 
  \Z_p[\Gal(K_n/k)](\chi_{\GG_n})) \right]^{\dual}\\
&\simeq \left[ \varinjlim_n H^i(k_v, (\Q_p/\Z_p) \otimes_{\Z_p} 
  \Z_p[\Gal(K_{n, w}/k_v)](\chi_{\GG_n})) \otimes_{\Z_p[G_{v,n}]} 
	\Z_p[\Gal(K_n/k)] \right]^{\dual}\\
&\simeq \left[ \varinjlim_n H^i(K_{n,w}, \Q_p/\Z_p) \right]^{\dual} 
   \otimes_{\RR_v} \RR\\
&\simeq H^i(K_{\infty, w}, \Q_p/\Z_p)^{\dual} \otimes_{\RR_v} \RR.
\end{align}
This implies the assertion for $i \neq 1$.
For $i = 1$, the above computation implies
\[
H^1(k_v, \T)
\simeq \left[ \varprojlim_n H^1(K_{n,w}, \Z_p(1)) \right] \otimes_{\RR_v} \RR.
\]
For each positive integer $m$, the exact sequence $0 \to \mu_{p^m} \to 
\overline{k_v}^{\times} \overset{(-)^{p^m}}{\to} \overline{k_v}^{\times} \to 0$ induces an isomorphism 
$K_{n, w}^{\times} / (K_{n, w}^{\times})^{p^m} \simeq H^1(K_{n, w}, \mu_{p^m})$.
By taking the inverse limit with respect to $m$ and $n$, we obtain 
$\varprojlim_n H^1(K_{n,w}, \Z_p(1)) \simeq J_w$.
This completes the proof.
\end{proof}

\begin{cor}\label{cor:25}
The complexes $\derR\Gamma(k_S/k, \T^{\dual}(1))^{\dual}$ and 
$\derR\Gamma(k_v, \T)$ for $v \nmid p$ are objects of $\DePT(\RR)$.
\end{cor}

\begin{proof}
Propositions \ref{prop:39} and \ref{prop:41} imply that these complexes 
are objects of $\DeP(\RR)$.
By Proposition \ref{prop:35}, the cohomology groups are torsion.
\end{proof}

     \subsection{The algebraic part of the proof}

We define a complex $C_S = C_S(K_{\infty}/k)$ 
as a mapping cone of 
$\bigoplus_{v \in S'} \derR\Gamma(k_v, \T) \to 
   \derR\Gamma(k_S/k, \T^{\dual}(1))^{\dual}[-2]$, namely 
define it such that
it fits into a distinguished triangle
\begin{equation}\label{eq:21}
\bigoplus_{v \in S'} \derR\Gamma(k_v, \T) \to 
   \derR\Gamma(k_S/k, \T^{\dual}(1))^{\dual}[-2] \to C_S \to,
\end{equation}
where the first morphism is induced by the restriction, using Proposition \ref{prop:40}.
By Corollary \ref{cor:25}, $C_S$ is actually an object of $\DePT(\RR)$.

\begin{prop}\label{prop:26}
We have $H^i(C_S) = 0$ unless $i = 1$, and an exact sequence
\begin{equation}\label{eq:20}
0 \to X_{S_p} \to H^1(C_S) \to Z_{S'}^0 \to 0
\end{equation}
of $\RR$-modules.
\end{prop}

\begin{proof}
Taking the long exact sequence associated to \eqref{eq:21} and using Proposition \ref{prop:35}, 
we obtain an exact sequence
\begin{align}
0 \to H^0(C_S) &\to J_{S'} \to X_S \to H^1(C_S)\\
& \to Z_{S'} \to \Z_p \to H^2(C_S) \to 0.
\end{align}
Then the assertion follows from the exact sequences \eqref{eq:23} 
and \eqref{eq:22}.
\end{proof}

\begin{cor}\label{cor:42}
The projective dimension of $H^1(C_S)$ is at most one, and we have
\begin{equation}\label{eq:27}
\Fitt_{\RR}(X_{S_p}) = \Fitt_{\RR}(H^1(C_S)) \Fitt_{\RR}^{[1]}(Z_{S'}^0).
\end{equation}
\end{cor}

\begin{proof}
Since $C_S$ is perfect, the first statement
of Proposition \ref{prop:26} tells us that 
$\pd_{\RR}(H^1(C_S)) < \infty$.
By the exact sequence \eqref{eq:20}, $H^1(C_S)$ does not contain 
any non-trivial finite submodule.
Hence we have $\pd_{\RR}(H^1(C_S)) \leq 1$.
The formula \eqref{eq:27} is therefore a consequence of \eqref{eq:20} 
and the definition of $\Fitt_{\RR}^{[1]}$.
\end{proof}

\subsection{Principality of $\Fitt_{\RR}(X_{S_p})$}\label{subsec:61} 

At the end of this section we put the preceding result
into perspective by discussing the exact conditions
under which the ideal
$\Fitt_{\RR}(X_{S_p})$ is principal. 
Keep the setup of preceding sections.

\begin{lem}\label{lem:47} 
Suppose that there is a place $v^* \in S'$ such that 
$\GG_{v*} \supset \GG_v$ for any $v \in S'$.
Then we have an isomorphism
\[
Z_{S'}^0 \simeq Z_{v^*}^0 \oplus \bigoplus_{v \in S', v \neq v^*}Z_{v}.
\]
\end{lem}

\begin{proof}
Put $Z_{S' \setminus \{v^*\}} = \bigoplus_{v \in S', v \neq v^*} Z_v$.
Consider the commutative diagram with exact rows and columns
\[
\xymatrix{
0 \ar[r] &
Z_{v^*}^0 \ar[r] \ar@{^{(}->}[d] &
Z_{v^*} \ar[r] \ar@{^{(}->}[d] &
\Z_p \ar[r] \ar@{=}[d] &
0 \\
0 \ar[r] &
Z_{S'}^0 \ar[r] \ar@{->>}[d] &
Z_{S'} \ar[r] \ar@{->>}[d] &
\Z_p \ar[r] &
0 \\
& Z_{S' \setminus \{v^*\}} \ar@{=}[r] &
Z_{S' \setminus \{v^*\}} &&
}
\]
We shall show that the left vertical sequence splits.
Pick any $v \in S'$ with $v \neq v^*$. Then since $\GG_{v^*} \supset \GG_v$, 
we have a natural surjective homomorphism 
\[
\pi_v: Z_v = \Z_p[\GG/\GG_v] \to \Z_p[\GG/\GG_{v^*}] = Z_{v^*}.
\]
Using these homomorphisms, define a homomorphism 
$s: Z_{S' \setminus \{v^*\}} \to Z_{S'}$ as follows.
For $x = (x_v)_{v \in S', v \neq v^*} \in Z_{S' \setminus \{v^*\}}$, put $s(x)_v = x_v$ if $v \neq v^*$ and put
\[
 s(x)_{v^*} = - \sum_{v \in S, v \neq v^*} \pi_v(x_v).
\]
Then define $s(x) = (s(x)_v)_{v \in S'} \in Z_{S'}$.
By construction, $s$ is a section of the natural projection 
$Z_{S'} \to Z_{S' \setminus \{v^*\}}$, and moreover the image 
of $s$ is contained in $Z_{S'}^0$.
Therefore $s$ gives a splitting of the left vertical sequence, 
which completes the proof.
\end{proof}

\begin{prop} \label{prop:63}
Suppose $K/k$ is a $p$-extension.
Put $S = S_p \cup S_{\ram}(K/k)$ and suppose that $S' = S \setminus S_p \neq \emptyset$ 
(note that this implies $K_{\infty} \neq k_{\infty}$).
Then the following are equivalent.
\begin{itemize}
\item[(i)] $\Fitt_{\RR}(X_{S_p})$ is a principal ideal.
\item[(ii)] $\pd_{\RR}(X_{S_p}) \leq 1$.
\item[(iii)] $\pd_{\RR}(Z^0_{S'}) \leq 1$.
\item[(iv)] $Z^0_{S'} = 0$.
\item[(v)] $S'$ consists of only one place $v^*$, and this
place satisfies $\GG_{v^*} =\GG$. In other words, $v^*$ must
be totally inert in $k_\infty/k$ and totally ramified in $K_\infty/k_\infty$. 
\end{itemize}
\end{prop}

\begin{proof}
The equivalence (i) $\Leftrightarrow$ (ii) follows from the argument of \cite[Proposition 4]{CiaGr98}.
The equivalence (ii) $\Leftrightarrow$ (iii) follows from the short exact sequence
(\ref{eq:20}) and the first line of Corollary \ref{cor:42}.
The equivalence (iv) $\Leftrightarrow$ (v) is clear, and the implication (iv) $\Rightarrow$ (iii) is trivial.

Now we show the implication (iii) $\Rightarrow$ (iv).
Put $H=\Gal(K_\infty/k_\infty)$, which is a non-trivial $p$-group by assumption.
We note that, for any $\RR$-module $M$ which is free of finite rank over $\Z_p$, we have $\pd_{\RR}(M) \leq 1$ if and only if $M$ is a free $\Z_p[H]$-module.

First suppose that all quotients $\GG/\GG_v$ with $v\in S'$ are non-trivial.
Then the $\Z_p$-rank of every $Z_v = \Z_p[\GG/\GG_v]$ is a $p$-power
$>1$, so that we have $\rank_{\Z_p}(Z^0_{S'}) \equiv -1 (\bmod p)$.
Hence $Z^0_{S'}$ cannot be free over $\Z_p[H]$.

Consequently, if $\pd_{\RR}(Z^0_{S'}) \leq 1$, then we have at least one $v^* \in S'$ such that $\GG/\GG_{v^*}$ is trivial.
Then by Lemma \ref{lem:47}, we obtain
$$ Z^0_{S'} \simeq \bigoplus_{v\in S', v \neq v^*} Z_v. $$
It is easy to check that, for each $v \in S'$ with $v \neq v^*$, we have $Z_v = \Z_p[\GG/\GG_v]$ is free over $\Z_p[H]$ if and
only if $\GG_v \cap H = 1$. But $\GG_v \cap H$ is the decomposition group
of (a prime above) $v$ in $K_\infty/k_\infty$, and by the assumption $S' = S_{\ram}(K/k) \setminus S_p$,
the prime $v$ must
ramify in $K_\infty/k_\infty$. Hence we must have $S' = \{v^*\}$.
\end{proof}

For completeness, we note the following.
\begin{lem}
Suppose that $K/k$ is a $p$-extension and that $S_{\ram}(K/k) \subset S_p$.
Then $\Fitt_{\RR}(X_{S_p})$ is a principal ideal if and only if $K_{\infty} = k_{\infty}$.
\end{lem}
\begin{proof}
As already used in the proof of Proposition \ref{prop:63}, the ideal $\Fitt_{\RR}(X_{S_p})$ is principal if and only if $\pd_{\RR}(X_{S_p}) \leq 1$.
By Propositions \ref{prop:39} and \ref{prop:35}, we see that $\pd_{\RR}(X_{S_p}) \leq 1$ is equivalent to $\pd_{\RR}(\Z_p) \leq 1$, which is true exactly when $K_{\infty} = k_{\infty}$.
\end{proof}


\section{Proof of main result (II)}\label{sec:76}

In this section, we complete the proof of Theorem \ref{thm:28} by 
determining $\Fitt_{\RR}(H^1(C_S))$. This is in a certain way the
arithmetic part of the proof. We need a few preliminaries concerning
determinants and Fitting ideals.

     \subsection{The determinant homomorphism}

This subsection is devoted to the homological 
algebra related to the determinant functor.
Let $\Inv(\RR)$ be the commutative group of invertible fractional 
ideals of $\RR$. We shall introduce a group homomorphism, called the 
determinant,
\[
\Det_{\RR}: K_0(\DePT(\RR)) \to \Inv(\RR).
\]
Here $K_0$ denotes the Grothendieck group of a triangulated category.
We refer to Knudsen-Mumford \cite{KM76} for more on the theory of 
determinants.

Let $\ChP(\RR)$ be the abelian category of perfect complexes 
of $\RR$-modules. More precisely, $\ChP(\RR)$ consists of bounded 
complexes $F$ of $\RR$-modules such that $F^i$ is finitely generated 
and projective for all $i$. Let $\ChPT(\RR)$ be the subcategory of 
complexes with torsion cohomology groups.

\begin{defn}
A graded invertible $\RR$-module is a pair $(L, \alpha)$ 
where $L$ is an invertible $\RR$-module and 
$\alpha:  \Spec(\RR) \to \Z$ is a locally constant map.
Two graded invertible $\RR$-modules $(L, \alpha)$ and $(L', \alpha')$ are said to be isomorphic if $\alpha = \alpha'$ and $L$ and $L'$ are isomorphic as $\RR$-modules.
For two graded invertible $\RR$-modules $(L, \alpha), (L', \alpha')$, 
we define 
\[
(L, \alpha) \otimes (L', \alpha') = (L \otimes_{\RR} L', \alpha + \alpha').
\]
Then $(\Hom_{\RR}(L, \RR), -\alpha)$ is the inverse of $(L, \alpha)$.
\end{defn}

\begin{defn}
For a finitely generated projective $\RR$-module $F$, let $\rank(F)$ 
denote the (locally constant) rank of $F$, and define the determinant 
of $F$ by
\[
\Det_{\RR}(F) = \left( \bigwedge_{\RR}^{\rank(F)} F, \rank(F) \right),
\]
which is a graded invertible $\RR$-module.
Let $\Det_{\RR}^{-1}(F)$ be the inverse of $\Det_{\RR}(F)$.
\end{defn}

\begin{lem}\label{lem:28}
The following statements hold true.

(1) Let $0 \to F' \to F \to F'' \to 0$ be an exact sequence of 
finitely generated projective $\RR$-modules.
Then we have a canonical isomorphism 
$\Det_{\RR}(F) \simeq \Det_{\RR}(F') \otimes \Det_{\RR}(F'')$.

(2) Let $F$ and $F'$ be finitely generated projective $\RR$-modules.
Then we have a canonical isomorphism
\[
\Det_{\RR}(F) \otimes \Det_{\RR}(F') \simeq \Det_{\RR}(F') 
   \otimes \Det_{\RR}(F),
\]
which is locally given by 
\[
a_1 \wedge \dots \wedge a_r \otimes b_1 \wedge \dots \wedge b_{r'} 
\mapsto (-1)^{r r'}b_1 \wedge \dots \wedge b_{r'} 
  \otimes a_1 \wedge \dots \wedge a_r.
\]
Here $r$ and $r'$ denote the local rank of $F$ and $F'$, respectively.
\end{lem}

The appearance of the sign is the reason of introducing the information 
of the rank in the definition of the determinant.

\begin{defn}
For each complex $F \in \ChP(\RR)$, we define its determinant by
\[
\Det_{\RR}(F) 
= \bigotimes_{i \in \Z} \Det_{\RR}^{(-1)^i}(F^i).
\]
Thanks to Lemma \ref{lem:28}(2), this is independent from the 
ordering of $\Z$.  We denote by $\Det_{\RR}^{-1}(F)$ its inverse.
\end{defn}

\begin{lem}\label{lem:29}
The following hold true.

(1) Let $0 \to F' \to F \to F'' \to 0$ be an exact sequence in $\ChP(\RR)$.
Then we have a natural isomorphism $\Det_{\RR}(F) \simeq 
\Det_{\RR}(F') \otimes_{\RR} \Det_{\RR}(F'')$.

(2) If $F$ is acyclic, then we have a natural isomorphism 
$\Det_{\RR}(F) \simeq (\RR, 0)$.

(3) Every quasi-isomorphism $F' \to F$ induces an isomorphism 
$\Det_{\RR}(F') \simeq \Det_{\RR}(F)$.
\end{lem}

\begin{proof}
(1) and (2) follow from Lemma \ref{lem:28}(1).

(3) Consider the mapping cone $F''$ of $F' \to F$.
Then we have an exact sequence $0 \to F \to F'' \to F'[1] \to 0$.
Since $F''$ is acyclic, (1) and (2) imply
\[
\Det_{\RR}(F'[1]) \otimes \Det_{\RR}(F) \simeq \Det_{\RR}(F'') \simeq (\RR, 0).
\]
Now the observation $\Det_{\RR}(F'[1]) \simeq \Det_{\RR}^{-1}(F')$ 
completes the proof.
\end{proof}

\begin{defn}\label{defn:62}
Suppose $F \in \ChPT(\RR)$.
Since $\Frac(\RR) \otimes_{\RR} F$ is acyclic, Lemma \ref{lem:29}(2) 
gives a natural isomorphism $\Det_{\Frac(\RR)}(\Frac(\RR) \otimes_{\RR} F) 
\simeq (\Frac(\RR), 0)$.  Therefore, we have a natural map
\[
\Det_{\RR}(F) \hookrightarrow \Det_{\Frac(\RR)} (\Frac(\RR) \otimes_{\RR} F) 
  \simeq \Frac(\RR).
\]
Here we disregard the degree since it is zero.
From now on, we identify $\Det_{\RR}(F)$ with its image in $\Frac(\RR)$.
This defines a mapping $\Det_{\RR}$ from the set of 
isomorphism classes of objects of
$\DePT(\RR)$ to the set of fractional ideals of $\RR$.
\end{defn}

\begin{lem}\label{lem:30}
The map $\Det_{\RR}$ that was just defined
 induces a group homomorphism $\Det_{\RR}: K_0(\DePT(\RR)) \to \Inv(\RR)$.
\end{lem}

\begin{proof}
This follows from Lemma \ref{lem:29} (1) and (3). 
\end{proof}

As a preparation for the main arguments, we now formulate 
two lemmas, relating determinants to Fitting ideals.

\begin{lem}\label{lem:31}
Let $F \in \DePT(\RR)$ be a complex and $n$ be an integer.
Suppose that we have $H^i(F) = 0$ for any $i \neq n$ and 
$\pd_{\RR}(H^n(F)) \leq 1$.  Let $\mathbb Q$ be the
foll ring of quotients of $\RR$ and let $\lambda
=\lambda_{F_{\mathbb Q}}$ be the canonical
trivialization $\Det_{\mathbb Q}(F_{\mathbb Q}) \to \mathbb Q$. 
Then we have
\[
\Fitt_{\RR}(H^n(F)) = \Det_{\RR}(F)^{(-1)^{n-1}}
\]
in $\Inv(\RR)$.
\end{lem}

\begin{proof}
By translation, we may and will assume that $n = 0$.
By using truncations, we see that the complex $F$ is quasi-isomorphic 
to the complex $H^0(F)[0]$.
Taking a projective resolution of $H^0(F)$ of length $2$, we can construct 
a perfect complex $F' \in \ChPT(\RR)$ which is quasi-isomorphic to $F$ 
such that $(F')^i = 0$ for $i \neq -1, 0$.
We can assume that both $(F')^{-1}$ and $(F')^0$ are free $\RR$-modules 
of the same rank $a$.

We take bases $e_1, \dots, e_a$ of $(F')^{-1}$ and $f_1, \dots, f_a$ of $(F')^{0}$.
Then we can identify the homomorphism $d: (F')^{-1} \to (F')^0$ 
with a matrix $A \in M_a(\RR)$, and we have $\Fitt_{\RR}(H^0(\RR)) = (\det(A))$.

On the other hand, one may verify $\Det_{\RR}(F')^{-1} = (\det(A))$.
This is a standard fact, but we give a sketch of the proof for completeness.
Let $f_1^*, \dots, f_a^*$ be the dual basis of $f_1, \dots, f_a$.
Put $\QQ = \Frac(\RR)$ for notational simplicity.
Then we have a natural isomorphism
\begin{align}
\Det_{\QQ}^{-1}(\QQ \otimes_{\RR} F') 
&= \bigwedge^a_{\QQ} (\QQ \otimes_{\RR} (F')^{-1}) \otimes_{\QQ} 
\Hom_{\QQ}\left(\bigwedge^a_{\QQ} (\QQ \otimes_{\RR} (F')^{0}) ,\QQ\right)\\
&\simeq \bigwedge^a_{\QQ} (\QQ \otimes_{\RR} (F')^{-1}) \otimes_{\QQ} 
\bigwedge^a_{\QQ} \Hom_{\QQ}((\QQ \otimes_{\RR} (F')^{0}) ,\QQ),
\end{align}
under which the trivialization $\Det_{\QQ}^{-1}(\QQ \otimes_{\RR} F') \simeq \QQ$ is given by
\[
(x_1 \wedge \dots \wedge x_a) \otimes (\varphi_1 \wedge \dots \wedge \varphi_a) \mapsto \det(\varphi_i(d(x_j)))_{i,j}
\]
for $x_1, \dots, x_a \in \QQ \otimes_{\RR} (F')^{-1}$ and $\varphi_1, \dots, \varphi_a \in \Hom_{\QQ}((\QQ \otimes_{\RR} (F')^{0}) ,\QQ)$.
Now
\[
\Det^{-1}_{\RR}(F') \simeq \bigwedge^a_{\RR} (F')^{-1} \otimes_{\RR} 
\bigwedge^a_{\RR} \Hom_{\RR}((F')^{0} ,\RR)
\]
 has $(e_1 \wedge \dots \wedge e_a) \otimes (f_1^* \wedge \dots \wedge f_a^*)$
as a basis over $\RR$ and it goes to
\[
\det(\varphi_i(d(x_j)))_{i,j} = \det(A)
\]
by the trivialization.
This proves $\Det_{\RR}(F')^{-1} = (\det(A))$.
\end{proof}

\begin{lem}\label{lem:34}
Let $F \in \DePT(\RR)$ be a complex and $n$ be an integer.
Suppose that we have $H^i(F) = 0$ for $i \neq n, n+1$ and $H^i(F)$ 
does not contain any nonzero finite submodule for $i = n, n+1$.
Then we have
\[
\Fitt_{\RR}(H^n(F)^*) = \Det_{\RR}(F)^{(-1)^{n+1}} \Fitt_{\RR}(H^{n+1}(F)),
\]
where the superscript $(-)^*$ denotes the Iwasawa adjoint.
\end{lem}

\begin{proof}
By translation, we may assume that $n = 0$.
By using truncations, we see that the complex $F$ is quasi-isomorphic 
to a complex $F'$ such that $(F')^i = 0$ for $i \neq 0, 1$.
Moreover, the construction of the truncations allows us to assume 
$(F')^1$ is a projective $\RR$-module.
Then we have an exact sequence
\[
0 \to H^0(F) \to (F')^0 \to (F')^1 \to H^1(F) \to 0.
\]
We can construct a projective $\RR$-module $\overline{F}$ and a homomorphism 
$\overline{F} \to (F')^0$ such that the composition $\overline{F} \to (F')^0 \to (F')^1$ 
is an injective homomorphism with torsion cokernel.
Then, by defining $P_1$ and $P_2$ as the cokernel of $\overline{F} \to (F')^0$ 
and $\overline{F} \to (F')^1$ respectively, we have an exact sequence
\[
0 \to H^0(F) \to P^0 \to P^1 \to H^1(F) \to 0.
\]
By the assumption, none of these modules contain any nonzero finite 
submodule. Then by the construction, we deduce that 
$\pd_{\RR}(P^i) \leq 1$ for $i = 0, 1$.
By a purely algebraic result (see \cite[Lemma 5]{BG03} or 
\cite[Remark 4.8]{Kata}), we have
\[
\Fitt_{\RR}(H^0(F)^*) = \Fitt_{\RR}(P^0)\Fitt_{\RR}(P^1)^{-1} \Fitt_{\RR}(H^{1}(F)).
\]

On the other hand, by construction, the complex $F$ is quasi-isomorphic 
to the complex $[P^0 \to P^1]$ located at degrees $0$ and $1$.
Hence the distinguished triangle 
\[
P^0[0] \to [P^0 \to P^1] \to P^1[-1] \to
\]
shows that
\[
\Det_{\RR}(F) = \Det_{\RR}(P^0[0]) \Det_{\RR}(P^1[-1])
= \Fitt_{\RR}(P^0)^{-1} \Fitt_{\RR}(P^1).
\]
The final equation follows from Lemma \ref{lem:31}.
This completes the proof.
\end{proof}

     \subsection{Description of $C_S$ by $p$-adic $L$-functions}

Let us now go back to the arithmetic situation. Recall that it is
our goal to determine the Fitting ideal of $H^1(C_S)$. We first express
it as a combination of determinants of one global and some local complexes.

\begin{lem}\label{lem:32}
We have
\[
\Fitt_{\RR}(H^1(C_S)) 
= \Det_{\RR}(\derR\Gamma(k_S/k, \T^{\dual}(1))^{\dual}) 
 \prod_{v \in S'} \Det_{\RR}(\derR\Gamma(k_v, \T))^{-1}.
\]
\end{lem}

\begin{proof}
By Lemma \ref{lem:31}, we have
\[
\Fitt_{\RR}(H^1(C_S)) = \Det_{\RR}(C_S).
\]
But the definition \eqref{eq:21} of $C_S$ and Lemma \ref{lem:30} imply
\[
\Det_{\RR}(C_S) 
= \Det_{\RR}(\derR\Gamma(k_S/k, \T^{\dual}(1))^{\dual}[-2]) 
 \prod_{v \in S'} \Det_{\RR}(\derR\Gamma(k_v, \T))^{-1}.
\]
This completes the proof, since the shift by $-2$ does not change the determinant.
\end{proof}

Now we deal with the global term in the preceding lemma.
The following is a formulation of an abelian equivariant 
main conjecture. 

\begin{thm}\label{thm:33}
Suppose the $\mu$-invariant of $X_{S_p}$ vanishes.
Then we have
\[
\Det_{\RR}(\derR\Gamma(k_S/k, \T^{\dual}(1))^{\dual}) = (\theta_S).
\]
\end{thm}

\begin{proof}
We use results of Ritter and Weiss; see \cite[\S 4]{RW04}. 
They constructed a certain exact sequence
\[
0 \to X_S(K_{\infty}) \to \Cok(\Psi) \to \Cok(\psi) \to \Z_p \to 0
\]
of finitely generated torsion $\RR$-modules with $\pd_{\RR}(\Cok(\Psi)) 
\leq 1, \pd_{\RR}(\Cok(\psi)) \leq 1$ (we do not give the definitions of $\Psi$ and 
$\psi$ here).
Moreover, the equality known as ``equivariant main conjecture'' 
\[
\Fitt_{\RR}(\Cok(\Psi)) \Fitt_{\RR}(\Cok(\psi))^{-1} = (\theta_S)
\]
is an established theorem (assuming the vanishing of the 
$\mu$-invariant), thanks to Ritter-Weiss. Note that even in the non-commutative case, Ritter and Weiss \cite{RiWe2011} and Kakde \cite{Kak2013} proved analogous results, again assuming the vanishing of the $\mu$-invariant.
 For the equivalence of the various
formulations of the equivariant main conjecture, see also 
\cite[Theorem 2.2]{Nic13}.

By Nickel \cite[Theorem 2.4]{Nic13}, the complex 
$\derR\Gamma(k_S/k, \T^{\dual}(1))^{\dual}$ is isomorphic 
in $\DePT(\RR)$ to the complex
\[
[\Cok(\Psi) \to \Cok(\psi)]
\]
located at degrees $-1, 0$.
Similarly as in the final paragraph of the proof of Lemma \ref{lem:34},
we have
\[
\Det_{\RR}([\Cok(\Psi) \to \Cok(\psi)]) 
  = \Fitt_{\RR}(\Cok(\Psi)) \Fitt_{\RR}(\Cok(\psi))^{-1}.
\]
This completes the proof.
\end{proof}

In preparation for the final part of the proof, we state
another lemma, which by now seems to be well known. 
Recall $\Gamma_K = \Gal(K_{\infty}/K)$ and put 
$\Lambda = \Z_p[[\Gamma_K]]$, which is a subring of $\RR = \Z_p[[\GG]]$.
For a prime ideal $\fq$ of $\Lambda$, let $\RR_{\fq}$ be the 
localization of $\RR$ with respect to the multiplicative set 
$\Lambda \setminus \fq$.

\begin{lem}\label{lem:43}
Let $f, g \in \RR$ be non-zero-divisors and $\II$ an ideal of $\RR$.
Suppose that $\II\RR_{p\Lambda} = \RR_{p\Lambda}$.
If $f\II = g\II$ holds, then $f\RR = g\RR$ holds.
\end{lem}

Now we compute the local contributions in Lemma \ref{lem:32}.

\begin{prop}\label{prop:36}
For every finite place $v$ of $k$ outside $p$, there exists 
a unique element $f_v \in \Frac(\RR)^{\times}$ satisfying the following.

(1) We have
\[
\Det_{\RR}(\derR\Gamma(k_v, \T)) = (f_v).
\]

(2) For any continuous character $\psi: \GG \to \overline{\Q_p}^{\times}$ 
such that $\psi|_{\GG_v}$ is non-trivial, we have
\[
\psi(f_v) =
\begin{cases}
	\frac{1 - \psi(\sigma_v)N(v)^{-1}}{1 - \psi(\sigma_v)} & 
	      \text{if $\psi$ is unramified at $v$;}\\
	    1 & 
			  \text{if $\psi$ is ramified at $v$.}
\end{cases}
\]
\end{prop}

\begin{proof}
Observe that property (2) ensures the uniqueness of $f_v$.

Let $\RR_v = \Z_p[[\GG_v]] \subset \RR$ and 
$\T_v = \Z_p(1) \otimes \RR_v(\chi_{\GG_v}^{-1})$, 
which is a local counterpart of $\T$.
Then $\derR\Gamma(k_v, \T)$ is induced by $\derR\Gamma(k_v, \T_v)$, so
\[
\Det_{\RR}(\derR\Gamma(k_v, \T)) = \Det_{\RR_v}(\derR\Gamma(k_v, \T_v)) \RR.
\]
This reduces the problem to a completely local statement.
We will in fact find $f_v$ in the ring $\Frac(\RR_v)$.
Fix a place $w$ of $K_{\infty}$ above $v$ so that 
$\GG_v = \Gal(K_{\infty, w}/k_v)$.

Put $n_v = \ord_p(N(v)-1) \geq 0$, which is the maximal integer 
such that $\mu_{p^{n_v}} \subset k_v^{\times}$.
Recall that $\TT_v$ is the inertia subgroup of $v$ in $\GG$.
Since $\TT_v$ is a quotient of $\OO_{k_v}^{\times}$ by local 
class field theory, the $p$-Sylow subgroup $\TT_v^{(p)}$ of $\TT_v$ 
is a cyclic $p$-group of order at most $p^{n_v}$.
Fix a generator $\delta_v$ of $\TT_v^{(p)}$.

We decompose $\GG_v$ into $\GG_v=\GG_{v}^{(p)} \times 
\GG_{v}^{(p')}$ such that $\GG_{v}^{(p)}$ is pro-$p$ and $\GG_v^{(p')}$ 
is of order prime to $p$. 
Then as in Subsection \ref{subsecDecGroupRing} we have 
$\RR_v=\Z_{p}[[\GG_{v}]]=\bigoplus_{\chi} {\mathcal O}_{\chi}
[[\GG_{v}^{(p)}]]$ 
where $\chi$ runs over equivalence classes of $p$-adic 
characters of $\GG_{v}^{(p')}$. 
We also decompose $\TT_v = \TT_{v}^{(p)} \times \TT_{v}^{(p')}$ 
where $\TT_{v}^{(p)}$ is pro-$p$ and $\TT_{v}^{(p')}$ is of order prime to 
$p$. 
Put $\TT_{v}'=\TT_{v}^{(p')}$ and $\RR_{v}'=\Z_p[[\GG_v/\TT_{v}']]$. 
Then we decompose $\RR_v=\Z_p[[\GG_v]]$ as
\begin{equation}\label{eq:37}
\RR_v = \RR_{v}' \times 
\prod_{\chi_{\mid \TT_{v}'} \neq 1} \RR_v^{\chi},
\end{equation}
where $\chi$ runs over the equivalent classes of characters 
of $\GG_{v}^{(p')}$, which are 
non-trivial on $\TT_{v}'$.
We define an element $f_v$ of $\Frac(\RR_v)$ such that
\[
f_v = \left( \frac{\delta_v - 1 + N_{\TT_v^{(p)}} 
  (1 - \sigma_v N(v)^{-1})}{\delta_v - 1 + N_{\TT_v^{(p)}} 
	  (1 - \sigma_v)}, (1)_{\chi_{\mid \TT_{v}'} \neq 1} \right)
\]
using the identification \eqref{eq:37}.

We shall show that this element satisfies the desired properties (1) and (2);
let us begin with the latter.

Property (2): Let $\psi: \GG_v \to \overline{\Q_p}^{\times}$ be a non-trivial continuous character.

First suppose that $\psi$ is unramified at $v$.
Then $\psi$ is trivial on $\TT_{v}'$, $\psi(\delta_v) = 1$, and 
$\psi(N_{\TT_v^{(p)}}) = \sharp \TT_v^{(p)}$.
Hence
\[
\psi(f_v) 
= \psi \left( \frac{\delta_v - 1 + N_{\TT_v^{(p)}} 
  (1 - \sigma_v N(v)^{-1})}{\delta_v - 1 + N_{\TT_v^{(p)}} (1 - \sigma_v)} \right)
= \frac{1 - \psi(\sigma_v) N(v)^{-1}}{1 - \psi(\sigma_v)}.
\]

Now suppose that $\psi$ is ramified at $v$.
If $\psi$ is non-trivial on $\TT_{v}'$, then $\psi(f_v) = 1$ by 
the definition of $f_v$.
Otherwise, $\psi$ is non-trivial on $\TT_v^{(p)}$ and hence we 
have $\psi(\delta_v) \neq 1$ and $\psi(N_{\TT_v^{(p)}}) = 0$.
Therefore
\[
\psi(f_v) 
= \psi \left( \frac{\delta_v - 1 + N_{\TT_v^{(p)}}  
 (1 - \sigma_v N(v)^{-1})}{\delta_v - 1 + N_{\TT_v^{(p)}} (1 - \sigma_v)} \right)
= 1.
\]

Property (1):  By Proposition \ref{prop:35} and Remark \ref{rem:37}, 
we can apply Lemma \ref{lem:34} to obtain
\[
\Fitt_{\RR_v}((J_w)^*) = \Det_{\RR_v}(\derR\Gamma(k_v, \T_v)) \Fitt_{\RR_v}(\Z_p).
\]
Note that we have $(J_w)^* \simeq J_w$ by the simple description 
in Remark \ref{rem:37}. By Lemma \ref{lem:43}, this formula 
characterizes $\Det_{\RR_v}(\derR\Gamma(k_v, \T_v))$.
Hence it is enough to show that
\begin{equation}\label{eq:38}
\Fitt_{\RR_v}(J_w) = f_v \Fitt_{\RR_v}(\Z_p).
\end{equation}

Recall the identification \eqref{eq:37}.
Since the actions of $\TT_{v}'$ on $\Z_p$ and $J_w$
are trivial, the equation \eqref{eq:38} for the $\chi$-part with 
$\chi_{\mid \TT_{v}'} \neq 1$ holds trivially. 
Thus we have only to worry about the 
$\RR_{v}'=\Z_p[[\GG_v/\TT_{v}']]$-component.

First we suppose $\mu_p \not\subset k_v^{\times}$, namely $n_v = 0$.
Since $\TT_v^{(p)}=1$, we have $\TT_v=\TT_{v}'$ and 
$\RR_{v}' = \Z_p[[\GG_v/\TT_v]]$, 
whose augmentation ideal is generated by $1 - \sigma_v$.
Then the $\RR_{v}'$-component of the equation \eqref{eq:38} says
\[
\frac{1 - \sigma_v N(v)^{-1}}{1 - \sigma_v} (1 - \sigma_v) = 
\begin{cases}
	(1) & (\mu_{p^{\infty}}(K_{\infty, w}) = 0)\\
	(1 - \sigma_v N(v)^{-1}) & (\mu_{p^{\infty}} \subset K_{\infty, w}^{\times})
\end{cases}
\]
as ideals.
Here we used Remark \ref{rem:37} and $\kappa_v(\sigma_v) = N(v)$ 
in the latter case.

We shall show that $\sigma_v - N(v)$ is a unit of 
$\RR_{v}'$ when 
$\mu_{p^{\infty}}(K_{\infty, w}) = 0$.
To do this, by Nakayama's lemma, it is enough to show 
$\chi(\sigma_v) - N(v) \in {\mathcal O}_{\chi}^{\times}$ 
for every character $\chi$ of $\GG_{v}^{(p')}$ which is trivial 
on $\TT_{v}'$.
Put $f=\#(\GG_{v}^{(p')}/\TT_{v}')$. 
Since $\chi(\sigma_v)^{f}=1$, it suffices to show 
$N(v)^f \not \equiv 1 \pmod{p}$.
Let $M$ be the maximum intermediate field of 
$K_{\infty, w}/k_v$ such that $M/k_v$ is a finite unramified extension 
of degree prime to $p$, 
so $\Gal(M/k_v)=\GG_{v}^{(p')}/\TT_{v}'$.
Our assumption $\mu_{p^{\infty}}(K_{\infty, w}) = 0$ implies
$\mu_p \not\subset M^{\times}$. 
Since the residue field of $M$ is $\F_{N(v)^f}$,
it follows that $N(v)^f \not \equiv 1 \pmod{p}$.

Next we suppose $\mu_p \subset k_v^{\times}$.
Take a lift $\ttilde{\sigma_v} \in \GG_v$ of $\sigma_v$.
Then we have
\begin{align}
\Fitt_{\RR_{v}'}(\Z_p) = (1 - \ttilde{\sigma_v}, \delta_v - 1)\\
\Fitt_{\RR_{v}'}(J_w) = (1 - \ttilde{\sigma_v} N(v)^{-1}, \delta_v - 1).
\end{align}
Hence the equation \eqref{eq:38} on $\RR_{v}'$ says
\[
(\delta_v-1 + N_{\TT_v^{(p)}} (1 - \sigma_v)) (1 - \ttilde{\sigma_v} N(v)^{-1}, \delta_v - 1)
= (\delta_v-1 + N_{\TT_v^{(p)}} (1 - \sigma_v N(v)^{-1})) (1 - \ttilde{\sigma_v}, \delta_v - 1).
\]
By $N_{\TT_v^{(p)}}(\delta_v - 1) = 0$, the each side is generated by $(\delta_v-1)^2$ and
\[
(\delta_v-1 + N_{\TT_v^{(p)}} (1 - \sigma_v)) (1 - \ttilde{\sigma_v} N(v)^{-1})
, \qquad (\delta_v-1 + N_{\TT_v^{(p)}} (1 - \sigma_v N(v)^{-1})) (1 - \ttilde{\sigma_v}),
\]
respectively.
Thus, it is enough to show that the difference of these elements, 
\[
\ttilde{\sigma_v} (1 - N(v)^{-1})(\delta_v-1),
\]
is contained in the ideal $(\delta_v-1)^2$.
By $\delta_v^{p^{n_v}} = 1$, we have
\[
0 = ((\delta_v-1) + 1)^{p^{n_v}} -1
\equiv p^{n_v} (\delta_v-1) \mod (\delta_v-1)^2.
\]
Since $N(v) - 1$ is an element of $p^{n_v}\Z_p$ by the definition of $n_v$, this implies $(N(v)-1)(\delta_v-1) \in (\delta_v-1)^2$.
This completes the proof.
\end{proof}


\subsection{Proof of Theorem \ref{thm:28}}   
%

By comparing the values given by the respective interpolation formulas, we have
\begin{equation}\label{eq:65}
\theta_S^{\modif} = \theta_S \prod_{v \in S'} f_v^{-1},
\end{equation}
where $f_v$ is introduced in Proposition \ref{prop:36}.
Therefore, Lemma \ref{lem:32}, Theorem \ref{thm:33}, and Proposition \ref{prop:36} imply
\begin{equation}\label{eq:64}
\Fitt_{\RR}(H^1(C_S)) = (\theta_S^{\modif}).
\end{equation}
Then Theorem \ref{thm:28} follows immediately from Corollary \ref{cor:42}.

\begin{rem}
Let us give a direct argument showing that the right hand side 
of Theorem \ref{thm:28} is actually independent of the choice of $S$.
Since this independence is a logical consequence of our main result, 
this verification is not strictly necessary, 
but we think that doing it anyway gives a nice consistence check 
for our result.

Let $S_1 \supset S$ be another finite set and put $S_1' = S_1 \setminus S_p$.
By the exact sequence
\[
0 \to Z_{S'} \to Z_{S_1'} \to \bigoplus_{v \in S_1 \setminus S} Z_v \to 0,
\]
we have an exact sequence
\[
0 \to Z_{S'}^0 \to Z_{S_1'}^0 \to \bigoplus_{v \in S_1 \setminus S} Z_v \to 0.
\]
Note that all $v \in S_1 \setminus S$ are 
unramified in $K_{\infty}$.
Since $\pd_{\RR}(Z_v) \leq 1$ for $v \in S_1 \setminus S$, we obtain
\[
\Fitt^{[1]}_{\RR}(Z_{S_1'}^0) = \Fitt^{[1]}_{\RR}(Z_{S'}^0) \prod_{v \in S_1 \setminus S} \Fitt^{[1]}_{\RR}(Z_v)
\]
(recall that $\Fitt^{[1]}_{\RR}$ is again a Fitting invariant 
by \cite[Theorem 2.6]{Kata}).
For $v \in S_1 \setminus S$, the description $Z_v = \RR / (1 - \sigma_v)$ shows
\[
\Fitt^{[1]}_{\RR}(Z_v) = (1 - \sigma_v)^{-1}.
\]
Hence, using Lemma \ref{int2}(1), we obtain
\[
\Fitt^{[1]}_{\RR}(Z_{S_1'}^0)\theta_{S_1}^{\modif} = \Fitt^{[1]}_{\RR}(Z_{S'}^0)\theta_{S}^{\modif}.
\]
This completes the proof of independence from the choice of $S$.
\end{rem}


\subsection{Integrality of $\theta_{S}^{\rm mod}$} \label{Integralityoftheta}

In this subsection we prove Theorem \ref{PropInt}.
%

We can immediately deduce Theorem \ref{PropInt} from the above 
equation \eqref{eq:64}.
However, the proof of \eqref{eq:64} relies on the deep result on the equivariant main conjecture (Theorem \ref{thm:33}), which is only established under $\mu = 0$.
For this reason, it is meaningful
to prove Theorem \ref{PropInt} directly.

Before the proof, we would like to insert a general remark.

\begin{rem}\label{rem:87}
Recall that $K/k$ is a finite abelian extension of totally real fields and $K_{\infty}$ is the cyclotomic $\Z_p$-extension of $K$.
Then it is easy to show that there is a finite abelian extension $K'$ of $k$ such that $K' \cap k_{\infty} = k$ and 
$K_{\infty}$ is the cyclotomic
$\Z_p$-extension of $K'$.
In fact, since the $p$-component of $\Gal(K_{\infty}/k)$ is a finitely generated $\Z_{p}$-module of rank $1$, 
it is decomposed into $\Z_{p} \times$ (a finite abelian $p$-group). 
Therefore, $\Gal(K_{\infty}/k)$ can be decomposed into $\Z_{p} \times$ 
(a finite abelian group), which proves the existence of $K'$.

We can see that many objects in this paper depend only on the extension $K_{\infty}/k$ and not on the specific field $K$.
In particular, the statements of Theorems \ref{thm:28} and \ref{PropInt} do not depend on $K$.
Therefore, we may assume to begin with that
\begin{equation}\label{AssumptionLinIndep}
K \cap k_{\infty}=k,
\end{equation}
whenever we are concerned with statements of this type.
\end{rem}

If $\chi$ is a non-trivial character of 
$\GG^{(p')}$, the integrality of the $\chi$-component of  
$\theta_{S}^{\modif}$, namely the fact that 
the image of $\theta_{S}^{\modif}$ 
is in ${\mathcal O}_{\chi}[[\GG_{v}^{(p)}]]$, can be easily chacked  
by a slight modification of the argument we give below. 
So we only consider the trivial character component in the following, and thus we may
assume $K/k$ is a $p$-extension.

We begin with the following lemma, which gives an alternative description of the element $f_v$ defined in Proposition \ref{prop:36}.
Note that we are now assuming $\GG^{(p')}=1$, so 
$\GG_{v}^{(p')}=1$ and $\RR_{v}=\RR_{v}'$ for all $v \in S'$, and we 
are interested only in the first component of $f_{v}$ in \eqref{eq:37}.

\begin{lem}\label{int1}
For a prime $v \in S'$, we take a lift $\ttilde{\sigma_v} \in \GG_v$ 
of $\sigma_v$. Then we have 
\[ 
f_{v}^{-1} = 
\frac{\ttilde{\sigma_v}-1 + (N(v)^{-1}-1)\cdot(\#\TT_v)^{-1}\cdot
   (\#\TT_v-N_{\TT_v})\ttilde{\sigma_v}}
{\ttilde{\sigma_v}N(v)^{-1}-1}.
\]
In particular, the right hand side does not depend on the choice 
of $\ttilde{\sigma_v}$.
\end{lem}

We note that the numerator of the right hand side of the equation in 
Lemma \ref{int1} is in $\RR$ because $N(v) \equiv 1 \pmod{\#\TT_v}$. 
Moreover, it is in the augmentation ideal of $\RR=\Z_p[[\GG]]$. 

\begin{proof}
We denote by $\alpha_v$ the right hand side.
Let $\psi$ be any character of $\GG$ and we compare the interpolation property of $\alpha_v$ with Proposition \ref{prop:36}(2).
If $\psi$ is non-trivial on $\TT_v$, then $\psi(N_{\TT_v}) = 0$ implies $\ttilde{\psi}(\alpha_v) = 1 = \ttilde{\psi}(f_{v}^{-1})$. 
If $\psi$ is trivial on $\TT_v$, then $\psi(N_{\TT_v}) = \sharp \TT_v$ implies
\[
\psi(\alpha_v) = 
\frac{1 - \psi(\sigma_v)}{1 - \psi(\sigma_v) N(v)^{-1}}
= \ttilde{\psi}(f_{v}^{-1}). 
\]
Thus we have $\alpha_v = f_v^{-1}$.
\end{proof}

We prepare an algebraic lemma.
For each character $\psi: H \to \overline{\Q_p}^{\times}$ of a finite abelian $p$-group $H$, put $\OO_{\psi} = \Z_p[\Image(\psi)]$.
Then $\psi$ induces a surjective ring homomorphism $\ttilde{\psi}: \Z_p[H][[T]] \to \OO_{\psi}[[T]]$, which extends to
$\ttilde{\psi}: \Frac(\Z_p[H][[T]]) \to \Frac(\OO_{\psi}[[T]])$.

\begin{lem}\label{lem:67}
Let $H$ be a finite abelian $p$-group.
Let $F(T), G(T) \in \Z_p[H][[T]]$ be elements such that the $\mu$-invariant of $F(T)$ is zero in the sense of \cite[\S 2]{BG03}.
Consider the element $\vartheta = G(T)/F(T)$ of $\Frac(\Z_p[H][[T]])$.
Then we have $\vartheta \in \Z_p[H][[T]]$ if and only if $\ttilde{\psi}(\vartheta) \in \OO_{\psi}[[T]]$ for any character $\psi$ of $H$.
\end{lem}

\begin{proof}
The ``only if'' part is clear, so we show the ``if'' part.

By the equivariant Weierstrass preparation theorem 
\cite[Proposition 2.1]{BG03}, $F(T)$ can be written as 
$F(T)=F^{{*}}(T)U(T)$ where $U(T)$ is a unit of $\Z_p[H][[T]]$ and 
$F^{{*}}(T)$ is a distinguished polynomial in $\Z_p[H][T]$. 
Then it is enough to show that $G(T)$ is divisible by $F^{{*}}(T)$.

Let $m$ be the degree of $F^{{*}}(T)$.
Since $F^{{*}}(T)$ is monic,
one can write $G(T)=F^{{*}}(T)q(T)+r(T)$ for some $q(T) \in \Z_p[H][[T]]$ and 
$r(T) \in \Z_{p}[H][T]$ such that the degree of $r(T)$ is less than $m$ (see also the ``generalized division lemma'' 
    in the proof of \cite[Proposition 2.1]{BG03}).
Thus we have only to show $r(T)=0$.  

Assume to the contrary that $r(T) \neq 0$. 
Then there is a character $\psi$ of $H$ such that $\ttilde{\psi}(r(T)) \neq 0$.
Since the degree of $\ttilde{\psi}(r(T))$ is less than $m$, we see that $\ttilde{\psi}(r(T))$ is not divisible by $\ttilde{\psi}(F^*(T))$.
This contradicts $\ttilde{\psi}(\vartheta) \in \OO_{\psi}[[T]]$.
Hence we have $r(T) = 0$.
\end{proof}

\begin{proof}[Proof of Theorem \ref{PropInt}]
Recall that we are assuming $\GG$ is a pro-$p$-group.
As we explained in Remark \ref{rem:87},
we may and do assume $K \cap k_{\infty} = k$.
Put $H = \Gal(K_{\infty}/k_{\infty})$, which is a finite abelian $p$-group by our assumption.
By fixing a topological generator of $\Gal(K_{\infty}/K)$, we may and do identify $\RR$ with $\Z_p[H][[T]]$.

By the formula \eqref{eq:65} and Lemma \ref{int1}, we have $\theta_S^{\modif} = G(T) / F(T)$ with
\[
F(T) = \prod_{v \in S'} (\ttilde{\sigma_v}N(v)^{-1}-1)
\]
and
\[
G(T) = \theta_{S} \prod_{v \in S'} \left(\ttilde{\sigma_v}-1+\frac{N(v)^{-1}-1}{\#\TT_v}(\#\TT_v-N_{\TT_v})\ttilde{\sigma_v}\right).
\]
Then we know that 
the $\mu$-invariant of $F(T)$ is zero in the sense of \cite[\S 2]{BG03}.
Moreover, since $S'$ is non-empty and $\theta_{S}$ is a pseudo-measure, 
we have $G(T) \in \Z_p[H][[T]]$.
Hence, by Lemma \ref{lem:67}, the assertion is equivalent to $\ttilde{\psi}(\theta_S^{\modif}) \in \OO_{\psi}[[T]]$ for any character $\psi$ of $H$.

We use induction on $\sharp H = [K_{\infty}: k_{\infty}]$ to show $\ttilde{\psi}(\theta_{S, K_{\infty}/k}^{\modif}) \in \OO_{\psi}[[T]]$.

First suppose that $S_{\ram}(K/k) \subset S_p$ (this includes the case where $K_{\infty} = k_{\infty}$).
Then we have
\[
\theta_S^{\modif} 
= \theta_{S_p}^{\modif} \prod_{v \in S'} (1 - \sigma_v)
= \theta_{S_p} \prod_{v \in S'} (1 - \sigma_v)
\]
by Lemma \ref{int2}(1).
Since $S' \neq \emptyset$ and $\theta_{S_p}$ is a pseudo-measure, we have $\theta_S^{\modif} \in \RR$.

Second suppose that $S_{\ram}(K/k) \not\subset S_p$ and $\psi$ is faithful as a character of $H$, so $H$ is cyclic in this case.
By Lemma \ref{int2}(1), we may assume that $S = S_{\ram}(K/k) \cup S_p$.
Since $\psi$ is then ramified at all $v \in S' = S_{\ram}(K/k) \setminus S_p$, we have
\[
\ttilde{\psi}(\theta_{S}^{\modif}) = \ttilde{\psi}(\theta_{S})
\]
by Definitions \ref{defn:44} and \ref{defn:67}.
Since $\psi$ is non-trivial and $\theta_{S}$ is a pseudo-measure, the right hand side is in $\OO_{\psi}[[T]]$.

Finally suppose that $\psi$ is not faithful as a character of $H$.
Put $M_{\infty} = K_{\infty}^{\Ker(\psi)}$, which is strictly smaller than $K_{\infty}$.
Then we see that $\ttilde{\psi}(\theta_{S, K_{\infty}/k}) = \ttilde{\psi}(\theta_{S, M_{\infty}/k})$ by Lemma \ref{int2}(2), which is in $\OO_{\psi}[[T]]$ by the induction hypothesis.
This completes the proof.
\end{proof}


\section{A strategy for computing $\Fitt^{[1]}_{\RR}(Z_{S'}^0)$}\label{sec:48}

In this section, we look at methods of computing $\Fitt^{[1]}_{\RR}(Z_{S'}^0)$.
The motivation for this is fairly obvious:
without any concrete information on this Fitting ideal, our main result
would remain rather abstract and impractical. 
As one application among others, we will reprove in Section \ref{sec:50} 
a previous result of the third author \cite{Kur11}.

Throughout this section, we assume that $K/k$ is a $p$-extension. 

     \subsection{The algebraic problem}\label{subsec:4.1}

We propose an algebraic problem whose full understanding
(if it can be achieved) will help a lot in computing 
$\Fitt^{[1]}_{\RR}(Z_{S'}^0)$.

Let $p$ be a prime number and $G$ a finite abelian $p$-group.
We denote the group ring by $R = \Z_p[G]$.
Take subgroups $G_1, \dots, G_r$ of $G$ with $r \geq 1$.
We consider the $R$-module
\[
Z = \bigoplus _{i=1}^r \Z_p[G/G_i]
\]
and the $R$-submodule $Z^0$ of $Z$, defined by the exact sequence
\begin{equation}\label{eq:a05}
0 \to Z^0 \to Z \to \Z_p \to 0,
\end{equation}
where the surjective map is the augmentation map.
Now the algebraic problem is the following.

\begin{prob}\label{prob:49}
How can we construct a free $R$-resolution of $Z^0$?
\end{prob}

In the subsequent sections, we will try to solve this problem.
Before that, we explain how to utilize a solution of Problem \ref{prob:49} 
for a computation of $\Fitt^{[1]}_{\RR}(Z_{S'}^0)$.

In the arithmetic situation as in Theorem \ref{thm:28}, let $K_n$ be the $n$-th layer of $K_{\infty}/K$.
Take $n$ 
sufficiently large such that no places in $S'$ split in $K_{\infty}/K_n$.
With this choice, put $G = \Gal(K_n/k)$ and let $G_v$ 
be the decomposition group of $v$ in $G$.
Then we can identify 
\[
Z_v = \Z_p[\GG/\GG_v] = \Z_p[G/G_v].
\]

Let us moreover put $R = \Z_p[G]$,
and note that $\pd_{\RR}(R) \leq 1$, since $R = \RR / (\gamma_K^{p^n}-1)\RR$ 
with $\gamma_K \in \Gamma_K = \Gal(K_{\infty}/K)$ a topological generator.

For a matrix $B$ over a commutative ring and a non-negative integer $e$, 
let $\Min_e(B)$ denote the ideal generated by the $e$-minors of $B$.

\begin{prop}\label{prop:53}
In the above situation, let 
\[
R^{t_3} \overset{A}{\to} R^{t_2} \to R^{t_1} \to Z_{S'}^0 \to 0
\]
 be an exact sequence over $R$.
We identify $A$ with a matrix and take a lift $\ttilde{A}$ of $A$ over $\RR$.
Then we have
\[
\Fitt^{[1]}_{\RR}(Z^0_{S'}) = (\gamma_K^{p^n}-1)^{t_2 - t_1} \sum_{e = 0}^{t_2} (\gamma_K^{p^n}-1)^{-e}\Min_e(\ttilde{A}). 
\]
\end{prop}

\begin{proof}
The short exact sequence
\[
0 \to \Cok(A) \to R^{t_1} \to Z^0_{S'} \to 0
\]
provides an equality
\[
\Fitt^{[1]}_{\RR}(Z_{S'}^0) = (\gamma_K^{p^n}-1)^{-t_1} \Fitt_{\RR}(\Cok(A)).
\]
Furthermore there is an exact sequence
\[
\RR^{t_3} \oplus \RR^{t_2} \overset{(\ttilde{A}, 
\gamma_K^{p^n}-1)}{\to} \RR^{t_2} \to \Cok(A) \to 0.
\]
By the definition of Fitting ideals, we obtain
\[
\Fitt_{\RR}(\Cok(A)) = \sum_{e = 0}^{t_2} (\gamma_K^{p^n}-1)^{t_2 - e}\Min_e(\ttilde{A}). 
\]
\end{proof}

   \subsection{How to attack Problem \ref{prob:49}: an idea}\label{subsec:a02}

We explain a very general idea which will be essential in the subsequent sections.
We shall construct
a homological complex $D$ of $R$-modules such that: 
\begin{enumerate}
\item[(a)] the components of $D$ are finitely generated free $R$-modules;
\item[(b)] $D$ is located in degrees $\geq 0$, that is, all components 
   in degree $\leq -1$ are zero (remember that the numbering is
	 homological, so the degrees increase when we go to the left);
\item[(c)] $D$ is exact except in degree $1$, and 
\[
H_1(D) \simeq Z^0.
\]
\end{enumerate}

We have to warn our readers right away that we have
to write the degrees as {\it superscripts,\/} not as subscripts
(which would be much more customary), in our homological complexes.
We will need the subscript position later, to distinguish different complexes
of similar type.

Such a complex $D$ gives a way to compute $\Fitt^{[1]}_{\RR}(Z^0_{S'})$ 
from a complex $D$ as follows.

\begin{prop}\label{prop:52}
Consider the arithmetic situation as in Proposition \ref{prop:53}. Let 
\[
D = [\dots \to D^3 \overset{A}{\to} D^2 \to D^1 \to D^0 \to 0]
\]
 be a complex over $R$ satisfying the above conditions (a)(b)(c).
Put $t_n = \rank_{R}(D^n)$ for $n \geq 0$.
We regard $A$ as a matrix over $R$ by choosing bases of $D^3$ and $D^2$, and take a lift $\ttilde{A}$ over $\RR$.
Then we have
\[
\Fitt^{[1]}_{\RR}(Z^0_{S'}) = (\gamma_K^{p^n}-1)^{t_2 - t_1 + t_0} \sum_{e = 0}^{t_2} (\gamma_K^{p^n}-1)^{-e}\Min_e(\ttilde{A}). 
\]
\end{prop}
\begin{proof}
By the properties (b) and (c), we have exact sequences
\[
0 \to \Ker(d_1) \to D^1 \overset{d_1}{\to} D^0 \to 0
\]
and
\[
\cdots \to D^3 \overset{A}{\to} D^2 \to \Ker(d_1) \to Z^0 \to 0.
\]
By the first sequence, the module $\Ker(d_1)$ is free of rank $t_1-t_0$.
Now Proposition \ref{prop:53} implies the assertion.
\end{proof}

In order to construct such a complex $D$, we will first construct complexes 
$C$ and $C_i$ ($i=1,\ldots,r$) which have similar properties.
More precisely, (a) and (b) will hold without change; and (c) is modified to 
(c$'$): $C$ and  $C_i$ are exact 
except in degree $0$, satisfying 
\[
H_0(C) \simeq \Z_p, \quad \text{and } H_0(C_i) = \Z_p[G/G_i]
\text{ for } i=1,\ldots,r.
\]
Moreover, we will construct a morphism of complexes 
\begin{equation}\label{eq:a08}
f: \bigoplus_{i=1}^r C_i \to C,
\end{equation}
which induces the augmentation homomorphism in degree $0$ homology.
Then, roughly speaking, $D$ can be constructed by either 
taking the mapping cone of $f$ or the cokernel of $f$; the choice
between these two options will depend on the precise setting.

   \subsection{The most general situation}\label{subsec:a09}

Let us describe a completely general method, even though
its usefulness is limited because it produces modules with far too large ranks.
The main ingredient is the standard resolution of finite groups,
which we recall now.

\begin{defn}
Let $G$ be a finite group.
For each $n \geq 0$, let $B_n(G)$ be the free $\Z_p[G]$-module on the set $\{(g_1, \dots, g_n) \mid g_1, \dots, g_n \in G\}$.
For $n \geq 1$, define a $\Z_p[G]$-homomorphism $B_n(G) \to B_{n-1}(G)$ by
\[
d_n((g_1, \dots, g_n)) = g_1(g_2, \dots, g_n) + \sum_{j=1}^{n-1} (-1)^j 
(g_1, \dots, g_j g_{j+1}, \dots, g_n) + (-1)^n (g_1, \dots, g_{n-1}).
\]
Moreover, define $\varepsilon: B_0(G) \to \Z_p$ by 
sending the empty tuple (which is by definition the only basis element of $B_0(G)$)
to $1$. 
\end{defn}

The following is well known.

\begin{prop}\label{prop:a06}
The sequence
\[
\dots \to B_2(G) \overset{d_2} \to B_1(G) \overset{d_1} \to B_0(G) 
  \overset{\varepsilon} \to \Z_p \to 0
\]
is exact.
\end{prop}

Therefore, the complex 
\[
C = [\dots \to B_2(G) \overset{d_2} \to B_1(G) \overset{d_1} \to B_0(G) \to 0]
\]
satisfies the conditions described above.
Similarly, for each $1 \leq i \leq r$, the complex
\[
C_i = [\dots \to B_2(G_i) \overset{d_2} \to B_1(G_i) \overset{d_1} \to B_0(G_i) \to 0] \otimes_{\Z_p[G_i]} \Z_p[G]
\]
satisfies the required conditions, because then 
$H_0(C_i) = \Z_p \otimes_{\Z_p[G_i]} \Z_p[G] = \Z_p[G/G_i]$.

For every $i\in\{1,\ldots,r\}$ there is a natural morphism $C_i \to C$ induced by
\[
\xymatrix{
\cdots \ar[r]
& B_2(G_i) \ar[r]^{d_2} \ar[d]
& B_1(G_i) \ar[r]^{d_1} \ar[d] 
& B_0(G_i) \ar[d] \ar[r]
& 0\\
\cdots \ar[r]
& B_2(G) \ar[r]_{d_2}
& B_1(G) \ar[r]_{d_1}
& B_0(G) \ar[r]
& 0
}
\]
where the vertical arrow in degree $n$ sends the basis element 
$(g_1, \dots, g_n) \in B_n(G_i)$ to the ``same'' basis element
$(g_1, \dots, g_n) \in B_n(G)$.
Thus we have a morphism $f$ as claimed in \eqref{eq:a08}.

Let $D = \Cone(f)$ be the mapping cone of $f$, so we have an exact sequence
\[
0 \to C \to D \to \bigoplus_{i=1}^r C_i[-1] \to 0.
\]
Then the conditions (a)(b)(c) for $D$ hold by construction.

     \subsection{A first special setting}\label{subsec:a14}

As we said above, the above construction tends to lead to
a free resolution of $Z^0$ with extremely unwieldy terms. Motivated
by this, we consider in this subsection the (fairly rare) case where
\[
G = G_1 \times \dots \times G_r
\]
and moreover $G_i$ is cyclic for each $i$.
In this case, we shall obtain an alternative construction of $C, C_i$, and $D$
involving much smaller modules.

\subsubsection{Definition of $C$}\label{subsubsec:a11}

The construction of $C$ will closely follow the
approach of the first and the third author in \cite{GK15}.
For all $i$, we choose a generator $\sigma_i$ of $G_i$, and
we denote by $N_{G_i}$ the norm element of $\Z_p[G_i]$.
Define a complex $E_i$ by
\begin{equation}\label{eq:a12}
E_i = [\dots \overset{\sigma_i -1}{\to} \Z_p[G_i] \overset{N_{G_i}}{\to} \Z_p[G_i]  \overset{\sigma_i -1}{\to} \Z_p[G_i] \to 0].
\end{equation}
Then $E_i$ is exact except for degree $0$, and $H_0(E_i) = \Z_p$.
We define
\[
C = E_1 \otimes_{\Z_p} \dots \otimes_{\Z_p} E_r,
\]
which satisfies the conditions (a)(b)(c), since $H_0(C) = \Z_p \otimes_{\Z_p} \dots \otimes_{\Z_p} \Z_p = \Z_p$.

The structure of $C$ is fully described in \cite{GK15}.
In particular, for each $n \geq 0$, the $n$-th component of $C$ 
is the free $R$-module on the set of monomials
\[
\{x_{l_1} \dots x_{l_n} \mid 1 \leq l_1 \leq \dots \leq l_n \leq r\}.
\]

\subsubsection{Definition of $C_i$}

For each $1 \leq j \leq r$, define
\[
E_j' = \Z_p[G_j][0] = [\dots \to 0 \to 0 \to \Z_p[G_j] \to 0].
\]
Thus, $\Z_p[G_j]$ is placed in degree zero. Define $C_i$ by
\[
C_i = E_1' \otimes_{\Z_p} \dots \otimes_{\Z_p} E_i \otimes_{\Z_p} \dots \otimes_{\Z_p} E_r'.
\]
(Here, only the $i$-th component is $E_i$, and all other components are $E_j'$.)
Then the conditions (a)(b)(c) hold because
\[
H_0(C_i) = \Z_p[G_1] \otimes_{\Z_p} \dots \otimes_{\Z_p} \Z_p \otimes_{\Z_p} \dots \otimes_{\Z_p} \Z_p[G_r]
= \Z_p[G/G_i].
\]

Note that the structure of $C_i$ is quite easy to understand. In a way,
$C_i$ arises from $E_i$ via base change from the smaller group ring $\Z_p[G_i]$
to the big group ring $\Z_p[G]$.
For each $n \geq 0$, the $n$-th component of $C_i$ is a free $R$-module of rank one,
and the differentials are ``the same'' as in the complex $E_i$.
The definition of the complexes $C_i$ is arranged in this particular way 
in order to make it possible to construct a map $f$ of complexes below.

\subsubsection{Definition of $f$}

For each $j \neq i$, we have a unique morphism $E_j' \to E_j$ which is identity in degree $0$.
Together with the identity morphism $E_i \to E_i$, we get a morphism $C_i \to C$.
Thus we obtain a morphism $f$ as claimed in \eqref{eq:a08}.

It is not hard to see that in degree $n$, the morphism $f$ sends the
canonical basis element of $C_i$ to the basis element $x_{i}^n$ of $C^n$.
This is a very special basis element, labeled by a power of one single variable
$x_i$; recall that the general basis element of $C^n$ is labeled 
by a general monomial 
of degree $n$ in $x_1,\ldots, x_r$.

\subsubsection{Definition of $D$}

We can take the mapping cone of $f$ to construct $D$ as in Section 
\ref{subsec:a09}. However, in our special case, it is much more 
efficient to consider the ``cokernel'' of $f$.
The quotation marks are supposed to draw attention to the minor 
problem that $f$ is not injective in degree $0$.
In fact, in degree $0$, the morphism $f$ looks like
\[
\Sigma: \bigoplus_{i=1}^r \Z_p[G] \to \Z_p[G].
\]
On the other hand, in all strictly positive degrees, the morphism 
$f: \bigoplus_{i=1}^r C_i \to C$ is fortunately injective 
and the cokernel is free over $R = \Z_p[G]$.
These facts can be read off from the description of $f$ just given,
in terms of the bases.

To avoid the minor problem in degree $0$, we modify $C$ to
\begin{equation}\label{eq:a10}
C' = C \oplus Y,
\end{equation}
where the acyclic complex 
$Y = [\dots \to 0 \to \bigoplus_{i=1}^r \Z_p[G] 
\overset{\id}{\to} \bigoplus_{i=1}^r \Z_p[G] \to 0]$ 
is concentrated in degrees $1$ and  $0$.
Then we can extend $f$ to an {\it injective} morphism 
$f': \bigoplus_{i=1}^r C_i \to C'$ such that the cokernel 
$D = \Cok(f': \bigoplus_{i=1}^r C_i \to C')$ satisfies the 
conditions (a)(b)(c). More precisely, we stipulate that in degree 0,
the new component of $f'$, that is, the additional
morphism of complexes $\bigoplus_{i=1}^{r} C_i \to Y$, 
is simply the identity morphism on $\bigoplus_{i=1}^{r} \Z_p[G]$.

This completes the construction of $D$.
Moreover the construction gives us a nice description of $D$ at no
expense at all. Indeed, the component $D^n$ of $D$ in any degree $n \geq 2$ 
is the free $R$-module on the set 
\[
\{x_{l_1} \dots x_{l_n} \mid 1 \leq l_1 \leq \dots \leq l_n \leq r \} \setminus \{x_1^n, \dots, x_r^n\}.
\]
(Here is a catch phrase describing this set: Take all monomials 
of degree $n$ and throw out the pure powers.)
The structure morphisms of $D$ are canonically induced by those of $C$.

\subsubsection{Arithmetic situation}\label{subsubsec:72}

Let us get back to the arithmetic situation.

We assume $K \cap k_{\infty} = k$ and also 
that every $v \in S'$ is inert in $K_{\infty}/K$.
By this assumption, we can put $G = \Gal(K/k)$ (see the text before Proposition \ref{prop:53}).
We label $S' = \{v_1, \dots, v_r\}$, put the decomposition groups $G_i = G_{v_i}$ for 
$i = 1, \dots, r$, and we suppose that $G = G_1 \times \dots \times G_r$ 
and each $G_i$ is cyclic.
Then, using the above information, it is possible to obtain much more precise
information on $\Fitt_{\RR}(Z^0_{S'})$.

\begin{eg}
Here is a modest numerical example. We take $p=3$, $k=\Q$ and
$K$ the compositum of the cubic extensions of $\Q$ with conductor
7 and 223 respectively. The primes 7 and 223 stay inert in $k_\infty$
because they are congruent to 1 modulo 3 but not modulo 9. 
We take $S'=\{7,223\}$. The group
$G=\Gal(K/\Q)$ is the product of the two decomposition groups at
7 and 223, since these two primes are cubic residues modulo each other.
This shows that we are indeed in the setting considered in this subsection
with $r = 2$.
We have not made any effort to determine the modified Stickelberger element
(even approximately), but this should be possible, in principle,
by going over to the minus side and using classical cyclotomic
Stickelberger elements.
\end{eg}

Returning to the general case, we construct the complex $D$ as above.
Then the ranks $t_n = \rank_R(D^n)$ satisfy $t_0 = 1, t_1 = r, t_2 = r(r-1)/2$.
Let $A$ be the matrix that describes the differential $d_3:
D^3 \to D^2$ in the complex $D$ constructed above, in the canonical
bases. Then the rows (columns) of $A$ are indexed by the monomials
in $x_1,\ldots,x_r$ of degree 3 (degree 2 respectively), with the
extra restriction that the pure powers $x_i^3$ ($x_i^2$ resp.) are omitted.
In \cite{GK15} and \cite{GKT19}, the differential $C^3 \to C^2$ was studied. It was
described by a matrix $\tilde M_r$, whose rows and columns were
indexed in exactly the same fashion, with the only difference
that the pure cubes and squares were still present as labels.
Since $D$ is obtained as a homomorphic image of $C$ as discussed
above, it is very simple to describe the new matrix $A$. It is obtained
from $\tilde M_r$ in \cite{GK15}, \cite{GKT19}
by eliminating all rows with labels $x_i^3$ and all
columns with labels $x_j^2$ (with $1\le i,j \le r$).  

The entries of $A$
are all of the form $\nu_i$ or $\tau_i$ (neglecting signs),
where we put $\nu_i = N_{G_i}$, the norm element, and $\tau_i = \sigma_i-1$ for compatibility with \cite{GK15}.
Note that $\tau_i\nu_i=0$.
Since we are assuming $K \cap k_{\infty} = k$, the natural map $H = \Gal(K_{\infty}/k_{\infty}) \to G = \Gal(K/k)$ is an isomorphism, 
so the matrix $A$ can be lifted to a matrix $\ttilde{A}$ uniquely as a matrix over $\Z_p[H]$.
We denote this matrix $\ttilde{A}$ simply by $A$.
Similarly by abuse of notation, we use the same symbols $\sigma_i, \nu_i, \tau_i$ in $H$, which are the canonical lifts from $G$.
Then the entries of the lifted matrix $A = \ttilde{A}$ have 
the same description in terms of $\nu_i$ and $\tau_i$.
Now by Proposition \ref{prop:52}, we have
\begin{equation}\label{eq:70}
  \Fitt^{[1]}_{\RR}(Z^0_{S'}) 
	 = T^{t_2 - r + 1} \sum_{e=0}^{t_2} 
	         T^{-e} \Min_{e}({A}),
\end{equation}
where we put $T = \gamma_K - 1$.

Let us first give examples where $r = 2$ or $r = 3$.
When we write presentation matrices in this section, we use the {\it row vector} convention. 

\begin{eg}\label{eg:68}
Suppose that $r=2$.
Then we have
\[
A  =  \left( \begin{array}{c}
  \nu_1\\
  \nu_2
      \end{array} \right),
\]
which shows that 
\[
\Fitt^{[1]}_{\RR}(Z^0) = (1, \nu_1/T, \nu_2/T).
\]
\end{eg}

\begin{eg}\label{eg:69}
Suppose $r = 3$.
We then get:    
\[
A  =  \left( \begin{array}{ccc}
  \nu_1 &&\\
  \nu_2 &&\\
        &\nu_1 &\\
        &\nu_3 &\\
        &         &\nu_2 \\
        &         &\nu_3 \\
\tau_3  & -\tau_2 & \tau_1
            \end{array} \right).
\]
Let us compute the 3-minors of $A$. To begin with, we
easily see that neither $\nu_{1}^2\tau_{2}$ nor $\nu_{1}^2\tau_{3}$ 
appears as a minor. But $\nu_{1}\nu_{2}\tau_{3}$ does appear. By a not overly
tedious complete verification, one obtains the following result.
Let 
\[
J=(\nu_{1}, \nu_{2}, \nu_{3}, \tau_{1},\tau_{2}, \tau_{3})
\]
be the ideal of $\RR$ generated by the given list of elements.
Then we find that 
\begin{align}
\Min_3(A)  &= (\nu_1 \nu_2, \nu_2\nu_3, \nu_3\nu_1)J;\\
 \Min_2(A) &= (\nu_1, \nu_2, \nu_3)J; \\
 \Min_1(A) &= J; \\
 \Min_0(A) &= (1).
\end{align}
Thus we obtain in the case $r=3$:
\[
\Fitt^{[1]}_{\RR}(Z^0) = T^{-2}(\nu_1 \nu_2, \nu_2\nu_3, \nu_3\nu_1)J 
    + T^{-1}(\nu_1,\nu_2,\nu_3)J + J + (T).
\]
\end{eg}

Let us now go back to general $r$.
In the situation of the paper \cite{GKT19}, which deals with the somewhat
larger matrix $\tilde M_r$ instead of $A$, the minors of $\tilde M_r$ are
completely determined for general $r$. 
Since $A$ is a submatrix of $\tilde M_r$, this
certainly gives an upper bound for $\Fitt^{[1]}_{\RR}(Z^0)$, which
is already something to start with.
We will be a little more ambitious and state a conjectural equality
in a moment. 

Our conjectural decription involves the notion of ``admissible''
$\nu$-monomials, which was already crucial in \cite{GKT19}. 
We quickly review the definition here and refer to that paper (\cite[Section 1.2]{GKT19}) for more information.

A $\nu$-monomial is, by definition, any expression $\nu_1^{f_1}\ldots \nu_r^{f_r}$
with exponents $f_1, \dots, f_r \geq 0$. 
Similarly, one defines $\tau$-monomials.
A $(\tau, \nu)$-monomial is, by definition, a product $z=xy$ with 
$x$ a $\tau$-monomial
and $y$ a $\nu$-monomial. We also demand that the events
$\tau_i$ appears in $x$, and $\nu_i$ appears in $y$ do not
happen simultaneously for any $i$ (because if that happens, we
get $z=0$). The $\nu$-part of such a $(\tau, \nu)$-monomial $z$ is simply $y$;
we write $\nu(z) = y$. 

\begin{defn}
We say that a $\nu$-monomial $y = \nu_1^{f_1}\ldots \nu_r^{f_r}$ is admissible if there is a permutation $\sigma$ of $\{1, 2, \dots, r\}$ satisfying
\[
f_{\sigma(1)} \geq f_{\sigma(2)} \geq \dots \geq f_{\sigma(r)}
\]
and
\[
\sum_{j = 1}^i f_{\sigma(j)} \leq \sum_{j=1}^i (r-j)
\]
for any $1 \leq i \leq r$.
\end{defn}

Putting it very roughly: to be an admissible monomial, the exponents must 
not be distributed too unevenly.

\begin{eg}
For $r=3$, a $\nu$-monomial of degree $3$ is admissible if
and only if it is not the cube of one $\nu_i$. 

Suppose $r=4$ and let us look at $\nu$-monomials of degree $6$.
Then for instance no exponent in an admissible monomial can exceed 3; but this does not suffice,
since for example $\nu_1^3\nu_2^3$ is not admissible either. 
On the other hand, $\nu_1^3\nu_2^2\nu_3$ is admissible, and many other
$\nu$-monomials as well.
\end{eg}

The following is proved in \cite{GKT19}.

\begin{thm}[{\cite{GKT19}}]
For all $e \geq 0$, the $e$-th minor ideal $\Min_d(\tilde M_r)$ is generated by all $(\tau, \nu)$-monomials of degree $e$
whose $\nu$-part is admissible. 
\end{thm}

To state our conjecture, a little extra notation will be useful. For
any $\nu$-monomial $y$, let $n(y)$ denote the number of indices $i$
such that $\nu_i$ occurs in $y$. Define $M(d,\ell)$ to be the set
of $(\tau, \nu)$-monomials $z$ of degree $d$ 
such that $\nu(z)$ is admissible and $n(\nu(z))\ge \ell$
(more simply put: which contain at least $\ell$ different $\nu_i$).
Finally, recall that $t_2=r(r-1)/2$.  

\begin{conj}
For all $e\ge 0$, the $e$-th minor ideal $\Min_e(A)$ is generated by
$M(e, r-1-t_2+e)$. 
\end{conj}

Note that for all $e\ge t_2-r-1$, the second argument
of $M(e,-)$ occurring here 
is non-positive, so the restricting condition concerning
the number of $\nu_i$ that must show up in the monomials
is vacuously satisfied. 

By the equation \eqref{eq:70}, this conjecture
implies the following statement. We could call it
the weak form of the above conjecture (which 
then would be called the strong form):

\begin{conj}\label{conj:71}
We keep the same notation and assumptions.
Then the union of the following sets generates
the ideal $\Fitt^{[1]}_{\RR}(Z^0_{S'})$:
\[
  T^{1-r}M(t_2,r-1), T^{2-r}M(t_2-1,r-2), T^{3-r}M(t_2-2,r-3),
	 \ldots, T^{t_2+1-r}M(0,r-1-t_2).
\]
Here the last exponent $t_2+1-r$ can be rewritten
$(r-1)(r-2)/2$, and the last set $M(0,r-1-t_2)$  only consists of
the trivial monomial 1.
\end{conj}    

It is not difficult to see that these conjectures agree with the result of our
calculations in the cases $r = 2$ (Example \ref{eg:68}) or $r=3$ (Example \ref{eg:69}). 
We think that we have verified it completely for $r=4$ as well. 
However, a general
proof for all $r$ would probably require to revisit most of the
technical arguments in \cite{GKT19}, and we haven't tried to do this.

We point out that
while the latter conjecture seems logically weaker than the former,
it is hard to image a proof (even partial)
of the weaker conjecture which does not pass through a proof of the 
strong one.

     \subsection{Another special setting}\label{subsec:a15}

In this subsection, apart from the setting in Subsection \ref{subsec:a14}, we suppose that $r=1$, so we are given only 
one subgroup $G_1 \subset G$.
In this case, we can decompose $G$ as an abelian group into
\[
G = G^{(1)} \times \dots \times G^{(s)}
\]
where all factors $G^{(j)}$ are non-trivial cyclic.
Hence $s$ is the $p$-rank of $G$.
Moreover, we can arrange the decomposition so that 
\[
G_1 = G_1^{(1)} \times \dots \times G_1^{(s)}
\]
with subgroups $G_1^{(j)} \subset G^{(j)}$.
Note that this is where we use $r=1$; in the general case, we cannot
expect a decomposition of $G$ into cyclic factors that is compatible
with all subgroups $G_i$. 

\begin{rem}\label{rem:73}
An important remark is in order. In our arithmetical setting for abelian $p$-extension $K/k$, $G_1$ is
always the decomposition group $G_v$ of a tamely ramified prime $v$ in some
abelian extension $K_n/k$ with group $G$. Since the ramification group
of such a prime is always cyclic, $G_v$ is always generated by 
at most two elements as an abelian $p$-group. Thus in the last formula
we can arrange things so as to have $G_1 = G_1^{(1)} \times G_1^{(2)}$. 
But in practice, this reduction is possibly not too helpful, see a
further comment below, where we do some calculations for the case $s=2$.
\end{rem}

\subsubsection{Definition of $C, C_1, f, D$}
Exactly as in Section \ref{subsubsec:a11}, 
we fix a generator $\sigma^{(j)}$ of $G^{(j)}$, for $1\le j \le s$. 
We  define a complex $E^{(j)}$ for every $j$ by \eqref{eq:a12}, with $G_i$ 
replaced by $G^{(j)}$, and a complex $C = E^{(1)} \otimes \dots \otimes E^{(s)}$.
Moreover, we may define $E_1^{(j)}$ exactly as $E^{(j)}$, just
replacing $G^{(j)}$ by $G_1^{(j)}$ and taking $(\sigma^{(j)})^{m_j}$ as the chosen generator of $G_1^{(j)}$, where we put $m_j = (G^{(j)}: G_1^{(j)})$.
Then we may put
$C_1 =  (E_1^{(1)} \otimes \dots \otimes E_1^{(s)}) \otimes_{\Z_p[G_1]} \Z_p[G]$.
It can be checked that $C$ and $C_1$ again satisfy the conditions (a)(b)(c).
(Note that there is only one value of the index $i$ now; $i=1$.)

We can define a morphism $f: E_1^{(j)} \to E^{(j)}$ explicitly.
We stipulate that $f$ is identity in every even degree, and 
$f$ is multiplication by $\mu^{(j)} =
1 + \sigma^{(j)} + \dots + (\sigma^{(j)})^{m_j-1}$ 
in every odd degree. We omit some easy verifications that certain
squares commute, making $f$ into a morphism of complexes. 

Hence we obtain a morphism $f: C_1 \to C$.
By taking the mapping cone of $f$, it is again possible to construct 
a complex $D$ satisfying the conditions (a)(b)(c).
 However, the
matrices representing the maps $C_1^n \to C^n$ in each degree $n$ will
contain entries involving factors $\mu^{(j)}$. Since these elements are in general
neither zero nor units in $R$, it cannot be expected (and indeed does
not happen in general) that the degree-wise cokernels of $f$ are again
free over $R$. So it does not seem possible to replace the cone of $f$
by the cokernel, and this makes the calculations more difficult.

\subsubsection{Arithmetic situation}\label{subsubsec:86}

As in Subsubsection \ref{subsubsec:72}, we assume 
$K \cap k_{\infty} = k$ and that every $v \in S'$ is inert in $K_{\infty}/K$.
Put $G = \Gal(K/k)$, which admits an isomorphism from $H = \Gal(K_{\infty}/k_{\infty})$.
Suppose that $S'$ consists of a single place $v_1$ and put $G_1 = G_{v_1}$.
Let $s$ be the $p$-rank of $G$.

\begin{eg}
First suppose $s = 1$; we omit all super- and subscripts $j$,
since $j=1$ is the only value. 
For example, $\sigma$ is a generator of $G$, $\tau = \sigma - 1$, and $m = (G: G_1)$.
Let $\ttilde{\nu}$ be the norm element of $G_1$.

One can check that 
the ranks $t_n$ of $D^n$ are given by $t_n = 2$ for $n \geq 1$ and $t_0 = 1$.
The differential from $D^3$ to $D^2$ is given by the square matrix
$\begin{pmatrix} \tilde\nu & 1 \cr 0 & \tau \end{pmatrix}$. Let $T = \gamma_K - 1$ be
defined as in Subsubsection \ref{subsubsec:72}. 
By Proposition \ref{prop:52} again,
we obtain:
\[
\Fitt^{[1]}_{\RR}(Z^0_{S'}) = (1, T^{-1}\ttilde{\nu}\tau). 
\]
For this case it is quite possible (and actually quicker) to calculate
the left hand side directly, by finding a simple resolution
of $Z^0$ by hand; we did it, and the results agree.
\end{eg}

\begin{eg}
Next suppose $s=2$. 
This case should be prototypical in a 
certain way, since (as pointed out in Remark \ref{rem:73}) we can always assume that
$G_1$ has only two cyclic summands $G_1^{(1)}$ and $G_1^{(2)}$;
of course $s$ can be larger than 2. Even letting $s=2$, we found the
calculation rather cumbersome. Again, we need to take the cone, not the
cokernel. 

The $R$-ranks $t_n$ of $D^n$ are given by 
$t_n = 2n + 1$ for $n \geq 0$. We determined
the matrix $A$ of the differential $D^3 \to D^2$. To write it down we
have to review notation:
For $j = 1, 2$, $\sigma^{(j)}$ is a generator of $G^{(j)}$ and $m_j = (G^{(j)}: G_1^{(j)})$; 
we put 
\begin{align}
\tau^{(j)} &= \sigma^{(j)}-1, \qquad \ttilde{\tau}^{(j)} = (\sigma^{(j)})^{m_j}-1,\\
\nu^{(j)} &= 1+\sigma^{(j)}+\dots+(\sigma^{(j)})^{n_j - 1}, \qquad
\ttilde{\nu}^{(j)} =
1 + (\sigma^{(j)})^{m_j} + \dots + (\sigma^{(j)})^{n_j-m_j},\\
\mu^{(j)} &= 1 + \sigma^{(j)} + \dots + (\sigma^{(j)})^{m_j-1},
\end{align}
where
$n_j=ord(\sigma^{(j)}) = \sharp G^{(j)}$.
 The outcome is the following 7 times 5 matrix:
\[
A = \begin{pmatrix} 
\ttilde{\nu}^{(1)} & & 1 \cr
\ttilde{\tau}^{(2)} & \ttilde{\nu}^{(1)} && \mu^{(1)}\mu^{(2)} \cr
  & \ttilde{\nu}^{(2)} &&& 1 \cr
	&& \tau^{(1)} \cr
	&& \tau^{(2)} & \nu^{(2)} \cr
	&&        & \nu^{(2)} & \tau^{(1)} \cr
	&&        &       &  \tau^{(2)}
\end{pmatrix}
\]
up to sign. 
Using ideals generated by the minors of this matrix $A$, it
is again possible to write down an expression that gives $\Fitt^{[1]}_{\RR}(Z^0)$.
But we have strong doubts whether such an expression would be
very enlightening. All 5-minors are 0 (unless we made an error), but
the minors of degree 4 and less seem to be so numerous and
hard-to-describe that it is
not really helpful to write them all down; in fact we gave up somewhere
along the way. 
\end{eg}


\section{Description of the Fitting ideal in the case that $\Gal(K/k)$ is cyclic}\label{sec:50}

In this section, we concentrate on cyclic $p$-extensions $K/k$ 
and we describe $\Fitt_{\RR}(X_{S_p})$ more explicitly from Theorem \ref{thm:28}.
We deduce two results from Theorem \ref{thm:28} 
(see Theorems \ref{thm:46} and \ref{thm:47}).
Both are a generalization of
the main result of the third author in \cite{Kur11} where the case 
$[K:k] = p$ was studied.
We keep the notations of Theorem \ref{thm:28}. 
As in Remark \ref{rem:87}, we will be assuming throughout $K \cap k_{\infty} = k$.
Moreover, we assume that
the $\mu$-invariant of $X_{S_p}$ vanishes.

For each finite place $v$ of $k$ outside $p$, 
as in Subsection \ref{subsecDecGroupRing}, 
$\nu_{v}: \Frac(\Z_p[[\GG/\TT_v]]) 
\to \Frac(\Z_p[[\GG]])$ is the map
induced by the multiplication by 
$N_{\TT_v}=\sum_{\sigma \in \TT_v} \sigma$.
Recall that $\sigma_v \in \GG/\TT_v$ is the Frobenius automorphism (Definition \ref{defn:74}).

Put $\Gamma_k = \Gal(k_{\infty}/k)$ and $H = \Gal(K_{\infty}/k_{\infty})$.
Then $N_H = \sum_{h \in H} h$ also 
induces a map 
$\Frac(\Z_p[[\Gamma_k]]) \to \Frac(\Z_p[[\GG]])$, which we denote by $\nu_H$.
Choose a topological generator $\gamma$ of $\Gamma_k$.

\begin{thm}\label{thm:46}
Suppose that $K/k$ is a cyclic $p$-extension
and $K \cap k_{\infty} = k$. 
Assume that we have a place $v^* \in S'$ such that $v^*$ is
totally ramified in $K/k$ and that $\GG_{v^*} \supset \GG_v$ holds for all $v\in S'$.
Then we have
\[
\Fitt_{\RR}(X_{S_p}) 
=  \left(1, \nu_{H}\frac{\gamma - 1}{\sigma_{v^*}-1} \right) 
  \prod_{v \in S', v \neq v^*} \left(1, \nu_v \frac{1}{\sigma_v-1} \right) 
   \theta_S^{\modif}.
\]
Note that, since $v^*$ is totally ramified in $K/k$, we have $\nu_H = \nu_{v^*}$.
\end{thm}

Before we prove this theorem, 
let us deduce a corollary which is exactly 
Theorem 0.1 (1) in \cite{Kur11}.

In general, for the cyclotomic $\Z_p$-extension 
$k_{\infty}/k$, let $k_n$ denote its $n$-th layer.
For each finite place $v$ of $k$ outside $p$, let $n_v$ be the non-negative integer 
such that the decomposition field of $v$ in $k_{\infty}/k$ is $k_{n_v}$.
This use of $n_v$
cancels an earlier meaning of $n_v$ in the proof of Proposition \ref{prop:36}.
Note that, if $v$ is totally ramified in $K_{\infty}/k_{\infty}$, then $\sigma_v$ is an element of $\Gamma_k = \Gal(k_{\infty}/k)$ which generates the same subgroup as $\gamma^{p^{n_v}}$.

\begin{cor}[{\cite[Theorem 0.1(1)]{Kur11}}]\label{thm:45}
Suppose that $K/k$ is a cyclic extension of degree $p$ 
and that $K \cap k_{\infty} = k$.
Take $S = S_p \cup S_{\ram}(K/k)$ and suppose that 
$S' = S \setminus S_p \neq \emptyset$
\footnote{While in Theorem \ref{thm:28} we may assume 
$S' \neq \emptyset$ by simply adding an arbitrary non-$p$-adic 
finite place, in Theorem \ref{thm:45} we have to take 
$S' = S_{\ram}(K/k) \setminus S_p$, so the condition 
$S' \neq \emptyset$ is a non-trivial restriction.}
.
Then we have
\[
\Fitt_{\RR}(X_{S_p}) 
= \sum_{v' \in S'} \left( \left(1, \nu_H \frac{\gamma - 1}{\sigma_{v'}-1} \right) \prod_{v \in S', v \neq v'} \left(1, \nu_H \frac{1}{\sigma_v-1} \right) \right) 
   \cdot \theta_S^{\modif}.
\]
\end{cor}

\begin{proof}
In this situation, any $v \in S'$ is totally ramified in $K/k$ since its degree is $p$.
Therefore, we have $\nu_{v}=\nu_{H}$.
Moreover, if we take $v^* \in S'$ such that $n_{v^*}$ is the minimum of the $n_v$, $v\in S'$, 
then $v^*$ satisfies the condition in Theorem \ref{thm:46}.
Then it is not hard to see that
\[
\sum_{v' \in S'} 
  \left( \left(1, \nu_H \frac{\gamma - 1}{\sigma_{v'}-1} \right) 
	\prod_{v \in S', v \neq v'} \left(1, \nu_H \frac{1}{\sigma_{v}-1} \right) \right)
= \left(1, \nu_H \frac{\gamma - 1}{\sigma_{v^*}-1} \right) 
 \prod_{v \in S', v \neq v^*} \left(1, \nu_H \frac{1}{\sigma_{v}-1} \right).
\]
Thus Theorem \ref{thm:46} implies the corollary.
\end{proof}

\begin{rem}\label{rem:51}
Here are several remarks on the equivalence of Theorem \ref{thm:45} 
and \cite[Theorem 0.1]{Kur11}.
\begin{itemize}
\item The statement of \cite{Kur11} concerns 
$(A_{K(\mu_{p^{\infty}})}^{\omega})^{\dual}$, where 
$A_{K(\mu_{p^{\infty}})}$ is the inductive limit of the $p$-parts 
of the ideal class groups of $K(\mu_{p^{n}})$ and $\omega$ denotes 
the Teichm\"{u}ller character.
Thanks to the Kummer duality between $A_{K(\mu_{p^{\infty}})}^{\omega}$ 
and $X_{S_p} = X_{S_p}(K_{\infty})$, we can translate the statement 
of \cite{Kur11} into a statement about $X_{S_p}$.
\item In the present paper, the hypothesis that the $\mu$-invariant
of $X_{S_p}$ vanishes
is slightly weaker than that of \cite{Kur11}.
\item The modified Stickelberger element $\vartheta_{K(\mu_{p^{\infty}})}$ 
is defined in \cite[(2.3.4)]{Kur11}, using a modifying factor $\xi$.
This factor $\xi$ corresponds to our modifying factor 
$\prod_{v \in S'} f_v^{-1}$.
\end{itemize}
\end{rem}




\begin{proof}[Proof of Theorem \ref{thm:46}]
By Theorem \ref{thm:28}, it is enough to show 
\[
\Fitt_{\RR}^{[1]}(Z_{S'}^0) 
  = \left(1, \nu_H \frac{\gamma - 1}{\sigma_{v^*}-1} \right) 
		  \prod_{v \in S', v \neq v^*} 
			   \left(1, \nu_v \frac{1}{\sigma_v-1} \right).
\]
By Lemma \ref{lem:47}, we have
\[
\Fitt_{\RR}^{[1]}(Z_{S'}^0) 
  = \Fitt_{\RR}^{[1]}(Z_{v^*}^0) 
	  \prod_{v \in S', v \neq v^*} \Fitt_{\RR}^{[1]}(Z_{v})
\]
since shifted Fitting ideals are multiplicative on direct sums.
The second term in the right hand side is computed in Proposition \ref{prop:85}.
For the first term, it is enough to show 
\[
\Fitt_{\RR}^{[1]}(\Z_p[\Gal(k_n/k)]^0) = \left(1, \nu_H \frac{\gamma - 1}{\gamma^{p^n}-1} \right)
\]
for any non-negative integer $n$.

In the rest of this calculation we will neglect some signs; 
this will play no role in the calculation of Fitting ideals via minors 
of certain matrices,
and it spares us the effort of being precise about the signs of morphisms
in the tensor product of complexes. These sign questions are certainly important
in many settings, but for us they are inessential and would mess up some arguments.

We apply the arguments of Sections 
\ref{subsec:a14} and \ref{subsec:a15} to 
\[
G = \Gal(K_n/k) = \Gal(K_n/k_n) \times \Gal(K_n/K).
\]
Take a generator $\delta$ of $H = \Gal(K_{\infty}/k_{\infty})$, 
which is identified with $\Gal(K_n/k_n) \subset G$.
Let $\gamma_K \in \Gal(K_{\infty}/K)$ be the lift of $\gamma \in \Gamma_k$.
As in Section \ref{subsec:a14}, define complexes
\[
C_1= [\dots \overset{\delta -1}{\to} \Z_p[G] \overset{N_H}{\to} \Z_p[G]  \overset{\delta -1}{\to} \Z_p[G] \to 0]
\]
and
\[
C = [\dots \overset{d_3}{\to} \Z_p[G]^3 \overset{d_2}{\to} \Z_p[G]^2  
 \overset{d_1}{\to} \Z_p[G] \to 0],
\]
where
\begin{align}
d_1 &= 
\begin{pmatrix}
\gamma_K-1 \\ \delta -1
\end{pmatrix},\\
d_2 &= 
\begin{pmatrix}
	N_n & 0 \\
	\delta - 1 & -(\gamma_K - 1)\\
	0 & N_H
\end{pmatrix},\\
d_3 &= 
\begin{pmatrix}
	\gamma_K - 1 & 0 & 0 \\
	\delta - 1 & -N_n & 0 \\
	0 & N_H & \gamma_K - 1 \\
	0 & 0 & \delta - 1
\end{pmatrix}
\end{align}
with $N_n = 1 + \gamma_K + \gamma_K^2 + \dots \gamma_K^{p^n-1}$
We have a natural injective homomorphism $C_1 \to C$ whose cokernel 
$D$ looks like
\[
\dots \to \Z_p[G]^3 \overset{d_3'}{\to} \Z_p[G]^2 \overset{d_2'}{\to} \Z_p[G] \to 0 \to 0,
\]
where $d_2', d_3', \dots$ are obtained by removing both the final row and the final column of $d_2, d_3, \dots$.
From this complex $D$, we obtain an exact sequence
\[
\RR^3 \oplus \RR^2 \overset{(\ttilde{d_3'}, \gamma_K^{p^n}-1)}{\to} \RR^2 \to \Z_p[G] \to \Z_p[\Gal(k_n/k)]^0 \to 0,
\]
where $\ttilde{d_3'}$ is any lift of $d_3'$.
Explicitly we can write down
\[
(\ttilde{d_3'}, \gamma_K^{p^n}-1) =
\begin{pmatrix}
	\gamma_K - 1 & 0 \\
	\delta - 1 & -N_n \\
	0 & N_H \\
	\gamma_K^{p^n}-1 & 0\\
	0 & \gamma_K^{p^n}-1
\end{pmatrix}
\]
and it is easy to see that the ideal generated by its $2 \times 2$ minors is
\[
( \gamma_K^{p^n}-1, N_H (\gamma_K-1)).
\]
Hence
\[
\Fitt_{\RR}^{[1]}(\Z_p[\Gal(k_n/k)]^0) 
= (\gamma_K^{p^n}-1)^{-1}( \gamma_K^{p^n}-1, N_H (\gamma_K-1))
= \left(1, \nu_H \frac{\gamma - 1}{\gamma^{p^n}-1} \right).
\]
This completes the proof.
\end{proof}

Finally there is another variant, where there is no privileged
place $v^*$ and still the degree of the cyclic $p$-extension $K/k$
can be arbitrary. The proof again uses the techniques
of Section 4; but this time the reduction lemma \ref{lem:47}
cannot be used. So one has to work with the full direct sum of ``local''
complexes $C_i$, not just one of them, and this makes it
inevitable to work with the cone, not the cokernel,
of the map $\bigoplus_i C_i \to C$ of complexes. This makes the
proof more complicated, but it will be given in very explicit terms.

\begin{thm}\label{thm:47}
Suppose that $K/k$ is a cyclic $p$-extension
with $K \cap k_{\infty} = k$ as before,
and let $\delta$ denote a generator of $H = \Gal(K_{\infty}/k_{\infty})$. 
Assume that the inertial degrees of all $v\in S'$ in $K/k$ are all 1
(that is, there is only ramification and splitting). 
Then $\Fitt^{[1]}_{\RR}(X_{S_p})$ is generated by the following
list of quantities:
  \[
   (\gamma-1,\delta-1) \prod_{v\in S'}\frac{\nu_v}{\sigma_v-1}\cdot \theta_S^{\modif}
  \]
and	
  \[
  \prod_{v\in J}\frac{\nu_v}{\sigma_v-1} \cdot \theta_S^{\modif},
  \]
where $J$ runs through all proper subsets of $S'$. The quantity
corresponding to $J=\emptyset$ is to be understood as 1.
\end{thm}

\begin{proof}
We use the lift $\gamma_K \in \Gamma_K$ of $\gamma \in \Gamma$ to define $T = \gamma_K - 1 \in \RR$.
We number $S'=\{v_1,...,v_r\}$ and put $\GG_i = \GG_{v_i}$, $\TT_i = \TT_{v_i}$, and $\nu_i = \nu_{v_i} = \nu_{\TT_i}$ for simplicity.
By the assumption that the inertial degree of $v_i$ in $K/k$ is one,
the decomposition field of $v_i$ in $K_{\infty}/k$ is an intermediate field of
the cyclotomic $\Z_p$-extension $(K_{\infty})^{\TT_i}/ K^{\TT_i}$.
Hence there is a unique lift $\ttilde{\sigma_i} \in \Gamma_K$ of $\sigma_i = \sigma_{v_i} \in \GG/\TT_i$.
Put $\delta_i = \delta^{[H: \TT_i]}$, which is a generator of $\TT_i$. 
Then $\GG_v$ is generated by $\delta_i$ and $\ttilde{\sigma_i}$.

We have two exact sequences involving $\Z_{p}$ and 
$Z_{S'} = \bigoplus_{i=1}^{r} \Z_{p}[\GG/\GG_i]$, aligning
in a commutative ladder as follows: 
\[
\xymatrix{
\cdots \ar[r]^-{(\delta_{i}-1)_i}
  &  \bigoplus_{i=1}^{r} \RR/(\ttilde{\sigma_i}-1) 
                \ar[r]^{(\nu_{i})_i}   \ar[d]^{\nat}
  &  \bigoplus_{i=1}^{r} \RR/(\ttilde{\sigma_i}-1) 
	              \ar[r]^{(\delta_{i}-1)_i}   \ar[d]^{\sum_{i} \mu_i}
	&  \bigoplus_{i=1}^{r} \RR/(\ttilde{\sigma_i}-1) \ar[r]	\ar[d]	^{\nat}
	& Z_{S'} \ar[r] \ar[d]^{\nat}
	& 0  \\
 \cdots \ar[r]_-{\delta-1}
  & \RR/(T) \ar[r]_{\nu_H} 
  & \RR/(T) \ar[r]_{\delta-1}
  & \RR/(T) \ar[r] 
	& \Z_{p} \ar[r] 
	&  0
}.
\]
The second and the fourth vertical arrows from the right are the canonical maps, whose
existence follow from the fact that $T$ divides $\ttilde{\sigma_i}-1$.
(This would not have worked without assuming that $v_i$
has no inertia in $K/k$.)  
The third vertical arrow from the right  on
the $i$-th component comes from multiplication by 
$\mu_i = 1 + \delta + \dots + \delta^{[H: \TT_i]-1}$.

By the cone construction, we obtain a complex $D$ of the form
\[
\bigoplus_{i=1}^{r} \RR/(\ttilde{\sigma_i}-1) \oplus \RR/(T)
\overset{d_3}{\to} \bigoplus_{i=1}^{r} \RR/(\ttilde{\sigma_i}-1) \oplus \RR/(T)
\overset{d_2}{\to} \bigoplus_{i=1}^{r} \RR/(\ttilde{\sigma_i}-1) \oplus \RR/(T)
\overset{d_1}{\to} \RR/(T)
\to 0,
\]
where the term $\RR/(T)$ is located at degree $0$.
Here,
\begin{align}
d_1(x_{1},...,x_{r},y) &= (x_{1} {\bmod} T,...,x_{r} \bmod T,
(\delta-1)y),\\
d_2(x_{1},...,x_{r},y) &= (-(\delta_1-1)x_{1},...,-(\delta_r-1)x_r, 
\sum_{i=1}^{r} \mu_i x_{i} \bmod T + \nu_{H}y),\\
d_3(x_{1},...,x_{r},y) &= (-\nu_{1}x_{1},...,-\nu_r x_r, 
\sum_{i=1}^{r} x_{i} \ \mbox{mod} \ T + (\delta-1)y).
\end{align}
Recall that the degree $n$ component of $D$ is denoted by $D^n$.
By a property of the cone, the complex $D$ is exact except in degree $1$, and $H_1(D) \simeq Z_{S'}^0$.
Then as in the proof of Proposition \ref{prop:52}, we have short exact sequences
\[
0 \to \Ker(d_1) \to D^1 \to D^0 \to 0
\]
and
\[
0 \to \Cok(d_3) \to \Ker(d_1) \to Z_{S'}^0 \to 0.
\]
Since we have $\pd_{\RR}(D^n) \leq 1$ for any degree $n$, the first exact sequence implies $\pd_{\RR}(\Ker(d_1)) \leq 1$ and
\[
\Fitt_{\RR}(\Ker(d_1)) 
= \Fitt_{\RR}(D^1) \Fitt_{\RR}(D^0)^{-1} 
= \prod_{i=1}^r (\ttilde{\sigma_i}-1).
\]
Then the second exact sequence implies
\begin{equation}\label{Proofthm:47Equation1}
\Fitt_{\RR}^{[1]}(Z_{S'}^0) 
= \Fitt_{\RR}(\Ker(d_1))^{-1} \Fitt_{\RR}(\Cok(d_3)) 
= \left(\prod_{i=1}^r (\ttilde{\sigma_i}-1)^{-1}\right) \Fitt_{\RR}(\Cok(d_3)).
\end{equation}

By the description of $d_3$ above, we easily see that the module $\Cok(d_3)$ has a presentation as an $\RR$-module
\[
B = 
\begin{pmatrix}
-\nu_{1} &  &        &          &     1   \\
 & -\nu_{2} &        &          &     1   \\
 &          & \ddots &          & \vdots  \\
 &          &        & -\nu_{r} &    1    \\
 &          &        &          & \delta-1 \\
\ttilde{\sigma_{1}}-1 &  &        &     &      \\
 & \ttilde{\sigma_{2}}-1 &        &     &      \\
 &               & \ddots &     &      \\
 &               &        & \ttilde{\sigma_{r}}-1 &   \\
 &               &        &               & T \\
\end{pmatrix}
\]
(all blank entries being zero). 
To obtain $\Fitt_{\RR}(\Cok(d_3))$, we have to  
compute the maximal minors of $B$.
As a consequence, we shall show that $\Fitt_{\RR}(\Cok(d_3))$ is generated by the following elements:
\begin{equation}\label{eq:75}
T\prod_{i = 1}^r \nu_i, \qquad (\delta-1) \prod_{i = 1}^r \nu_i, 
\qquad \prod_{i \in J} \nu_i \prod_{i \not \in J} (\ttilde{\sigma_i}-1)
\end{equation}
where $J$ runs through all proper subsets of $\{1, 2, \dots, r\}$.

Let $V$ run through all subsets of $\{1, 2, \dots, 2r+2\}$ of cardinality $r+1$, 
and let $d_V$ be the determinant of the submatrix of $B$ picking up the $v$-th rows for $v \in V$ (we ignore the sign throughout).
For each $1 \leq i \leq r$, the $i$-th column in $B$ is zero except for the $i$-th and $(r+1+i)$-th rows.
Hence $d_V \neq 0$ only if, for each $1 \leq i \leq r$, we have either $i \in V$ or $r+1+i \in V$.
We divide the argument into two cases.

Case 1. There is (a unique) $1 \leq l \leq r$ such that both $l \in V$ and $r+1+l \in V$ hold.
In this case, putting $J = (V \cap \{1, 2, \dots, r\}) \setminus\{l\}$, we can see that
\[
d_V = \pm \prod_{i \in J} \nu_i \prod_{i \not \in J} (\ttilde{\sigma_i}-1),
\]
where $i \not \in J$ means $i$ runs through $\{1, 2, \dots, r\} \setminus J$.
In this way, we obtain the third family of elements in \eqref{eq:75}.

Case 2. For each $1 \leq i \leq r$, exactly one of $i \in V$ or $r+1+i \in V$ holds.
Put $J = \{1, 2, \dots, r\} \cap V$.
Then exactly one of $r+1 \in V$ or $2r+2 \in V$ holds, and accordingly we obtain
\[
d_V = \pm (\delta - 1) \prod_{i \in J} \nu_i \prod_{i \not \in J} (\ttilde{\sigma_i}-1),
\qquad d_V = \pm T \prod_{i \in J} \nu_i \prod_{i \not \in J} (\ttilde{\sigma_i}-1).
\]
These elements are in the ideal generated by \eqref{eq:75}.
Moreover, letting $V = \{1, 2, \dots, r, r+1\}$ and $V = \{1, 2, \dots, r, 2r+2\}$ respectively, we can produce the first two elements in \eqref{eq:75} among $d_V$.

We obtain Theorem \ref{thm:47} from \eqref{Proofthm:47Equation1}, \eqref{eq:75}, 
and Theorem \ref{thm:28}.
\end{proof}

{
\bibliographystyle{acm}
\bibliography{fitting_ref}
}

\end{document}